\documentclass[10pt]{amsart}

\usepackage{amsmath,amsfonts,amsthm}
\usepackage{color}
\usepackage[]{graphics}
\usepackage{verbatim}
\usepackage{eepic,epic}
\usepackage{epsfig,epstopdf}
\usepackage[colorlinks,linktocpage,linkcolor=blue]{hyperref}
 \usepackage{caption}
  \usepackage{subcaption}
\allowdisplaybreaks

\newtheorem{theorem}{Theorem}[section]
\newtheorem{lemma}{Lemma}[section]

\newtheorem{assumption}[theorem]{Assumption}

\numberwithin{equation}{section}
\newcommand{\al}{\alpha}
\newcommand{\fy}{\varphi}
\renewcommand{\d}{{\,\rm d}}
\def\Dal{{\partial_t^\al}}

\def\Om{\Omega}
\def\II{(\Om)}

\def\dH#1{\dot H^{#1}(\Omega)}

\def\bDal{\bar\partial_\tau^\alpha}

\def\L2Om{{L^2(\Omega)}}

\def \E{E}

\def\tu{{u_\gamma}}
\def\tud{{u_\gamma^\delta}}

\def\R{\mathbb{R}}
 \def\tU{\tilde{U}}
  \def\tuh{u_{\gamma,h}}

\title[Backward problem for subdiffusion with time-dependent coefficients]{Stability and numerical analysis of backward problem for subdiffusion with time-dependent coefficients}

 \author[Zhengqi Zhang]{$\,\,$Zhengqi Zhang$\,\,$}
\address{Department of Applied Mathematics,
The Hong Kong Polytechnic University, Kowloon, Hong Kong.}
\email {19076082r@connect.polyu.hk}

\author[Zhi Zhou]{$\,\,$Zhi Zhou$\,$}
\address{Department of Applied Mathematics,
The Hong Kong Polytechnic University, Kowloon, Hong Kong}
\email {{zhizhou{\it @\,}polyu.edu.hk}}


\begin{document}

\maketitle

\begin{abstract}
Our aim is to study the backward problem, i.e. recover the initial data from the terminal observation, of the 
subdiffusion with time dependent coefficients. First of all, by using the smoothing property of solution operators
and a perturbation argument of freezing the diffusion coefficients, we show a stability estimate in Sobolev spaces, under some smallness/largeness condition on the terminal time.
Moreover, in case of noisy observation, we apply a quasi-boundary value method to regularize the problem and then show the convergence of the regularization scheme.
Finally, to numerically reconstruct the initial data, we propose a completely discrete scheme by applying the finite element method in space and backward Euler 
convolution quadrature in time. 
An \textsl{a priori} error estimate is then established. The proof is heavily built on a 
perturbation argument dealing with time dependent coefficients
and some nonstandard error estimates for the direct problem. 
The error estimate gives a useful guide for balancing discretization parameters, regularization parameter 
and noise level. Some numerical experiments are presented to illustrate our theoretical results.


\end{abstract}

\section{Introduction}
Let $\Omega\subset\mathbb{R}^d $ ($d=1,2,3$) be a
convex polyhedral domain
with boundary $\partial\Omega$. We are interested in the fractional evolution model with time-dependent coefficient:
\begin{align}\label{eqn:fde-t}
\begin{aligned}
\partial_t^\alpha u (x,t)+ \nabla \cdot(a(x,t)\nabla u) &= f(x,t),\ &&\text{in}\ \Omega\times(0,T],\\
u(x,t) &= 0,\ &&\text{on}\ \partial\Omega, \\
u(x,0)& = u_0(x),\ &&\text{in}\ \Omega, \\
\end{aligned}
\end{align}
where $T>0$ is a fixed final time, $f \in L^\infty(0,T;L^2(\Omega))$ and $u_0\in L^2(\Omega)$ are given
source term and initial data, respectively. $a(x,t)\in\R^{d\times d}$ is a symmetric matrix-valued
diffusion coefficient such that for  constants $c_0\geq 1$ and $c_1>0$
\begin{align}
  &c_0^{-1}|\xi|^2\le a(x,t)\xi\cdot\xi \le c_0 |\xi|^2, && \forall\, \xi\in\R^d, \,\, \forall\, (x,t)\in \Omega\times\mathbb{R}^+, \label{Cond-1-t} \\
  &|\partial_t a(x,t)|+|\nabla_xa(x,t)|+|\nabla_x\partial_t a(x,t)|\le c_1, &&\forall\, (x,t)\in \Omega\times\mathbb{R}^+. \label{Cond-2-t}
\end{align}
Here $\cdot$ and $|\cdot|$ denote the standard Euclidean inner product and norm, respectively, and $\mathbb{R}^+ = [0,\infty)$.
In \eqref{eqn:fde-t},
$\Dal u(t)$ denotes the Caputo fractional derivative in time $t$ of order $\alpha\in(0,1)$ \cite[p. 70]{KilbasSrivastavaTrujillo:2006}
\begin{align*}
   \Dal u(t)= \frac{1}{\Gamma(1-\alpha)}\int_0^t(t-s)^{-\alpha} \tfrac{\partial}{\partial s} u(s)\d s.
\end{align*}
where $\Gamma(z) = \int_0^\infty s^{z-1}e^{-s}\d s$(for $\Re(z)>0$) denotes the Euler's Gamma function. As the order $\al\to 1^-$, the fractional derivative converge to the standard first derivative $u'$ for sufficiently smooth $u$. 
Due to its numerous applications in physics, engineering, biology, and finance, there has been a surge in interest in fractional / nonlocal models in recent years. In particular, the time-fractional diffusion equation models the mean squared displacement of particles grows only sublinearly with time, compared to classical diffusion process growing linearly in time. 
The model are extensively used to describe subdiffusion process in nature, such as highly
heterogeneous aquifers in media and fractal geometry.
For interested readers, see a long list of applications of fractional models discovered from biology and physics in \cite{Metzler:2014,MetzlerKlafter:2000}.

Inverse problems for subdiffusion models have been extensively studied and there has already been a vast literature; we recommend review papers \cite{JinRundell:2015, LiLiuYamamoto:2019b, LiYamamoto:2019a,LiuLiYamamoto:2019} and reference therein. In this paper, we focus on backward problem for the subdiffusion model \eqref{eqn:fde-t}: to recover
the  initial data $u_0(x)$ with $x\in\Omega$  from terminal observation 
\begin{equation*}
u(x,T) = g(x),~~\text{for all }x\in \Omega.
\end{equation*}

In recent years, the backward subdiffusion problem has received a lot of attention.
All existing works are for the case that the coefficients are independent of time, a.e. $a(x,t) \equiv a(x)$.
The uniqueness and some stability estimate can be found in the pioneer work \cite{SakamotoYamamoto:2011}. See  also
\cite{LiuYamamoto:2010,WangLiu:TV,YangLiu:2013,WeiWang:2014,Nho_H_o_2019,Tuan:2020,WeiXian:2019}
for various kinds of regularization methods, and \cite{ZhangZhou:2020} for error analysis of numerical reconstruction by fully discrete schemes.
The analysis in aforementioned works heavily relies on the asymptotic behaviors of Mittag--Leffler functions, or equivalently the smoothing properties
of solution operators. 
Unfortunately, this strategy cannot be directly extended to
subdiffusion models with time dependent coefficients. 
Note that the backward problem for the parabolic equation with time-dependent coefficient 
has been intensively studied. 
However, the studies of fractional diffusion model \eqref{eqn:fde-t} is much more challenging
since many useful mathematical tools, including product rule and chain rule, are not directly applicable. 

There have been some existing works for the direct problem of time-fractional model \eqref{eqn:fde-t}.
For time-dependent elliptic operators or nonlinear problems, energy arguments
\cite{VergaraZacher:2015} or perturbation arguments \cite{KimKimLim:2017} can be used to show
existence and uniqueness of the solution. However, more refined stability estimates, needed for
numerical analysis of nonsmooth problem data, often have to be derived separately.
Mustapha \cite{Mustapha:2018} analyzed the spatially semidiscrete Galerkin FEM approximation of problem \eqref{eqn:fde-t} 
using a novel energy argument, and established optimal-order convergence rates for both smooth and nonsmooth initial data. 
See also \cite{McLeanMustaphaAliKnio:2019fcaa,McLeanMustaphaAliKnio:2020,McLeanMustaphaAliKnio:2020b} for time-fractional advection diffusion equation. 
In \cite{JinLiZhou:2019}, a perturbation argument of freezing the diffusion coefficients was proposed to analyze the model \eqref{eqn:fde-t} and its numerical treatment.
The argument was then modified and adapted to the error analysis of high-order discretization scheme in \cite{JinLiZhou:2020nm}.
However, the analysis for the backward problem is still missing in the literature. 

The first contribution of the paper is to develop a conditional stability of the backward problem in Sobolev spaces. 
Under some assumptions on diffusion coefficients and some smallness/largeness conditions
on terminal time $T$, there holds the Lipschitz stability (Theorems \ref{thm:stab-small} and \ref{thm:stab-big})
\begin{equation}\label{eqn:stab-intro}
 \|  u_0  \|_{L^2\II} \le c(1+T^\alpha) \|  u(T) \|_{H^2\II}. 
\end{equation}
The proof heavily relies on several \textsl{a priori} estimates of the direct problem,
 the smoothing properties of solution operators with frozen diffusion coefficients,
and a perturbation argument. The stability estimate \eqref{eqn:stab-intro} plays a key role in the 
analysis of regularization scheme and completely discrete approximation.

In practice, the observational data often involves random noise. In this work, we denote the empirical observation by $g_\delta$
and assume that it is noisy with a level $\delta>0$ in the sense that
\begin{equation}\label{eqn:noisy data}
\|g_\delta-g\|_{L^2\II} = \delta.
\end{equation}
According to the stability estimate \eqref{eqn:stab-intro}, we know the backward problem is mildly ill-posed,
and it is necessary to apply some regularization in case of noisy data.
In this work, we propose a quasi-boundary method for regularization
and analyze convergence of the regularized solution.
Specifically, under the assumptions in the stability estimate \eqref{eqn:stab-intro},
if $\| u_0 \|_{\dH q}\le c$ with $q\in (0,2]$, we prove that (Theorem \ref{thm:err-reg})
\begin{equation*}
 \| \tud(0) - u_0  \|_{L^2\II} \le c \,\Big(\delta \gamma^{-1} +  \gamma^{\frac q2}\Big).
\end{equation*}
Here $\gamma$ denotes the regularization parameter and $\tud(0)$ denotes the reconstruction via the 
regularization scheme. Then the optimal convergence rate (in terms of noise level)
is $O(\delta^{\frac{q}{q+2}})$ provided that $\gamma \sim \delta^{\frac2{q+2}}$.
Moreover, for $u_0\in L^2\II$, there holds
\begin{equation*}
\| \tud(0) -  u_0\|_{L^2\II} \rightarrow 0\quad \text{as}~~\delta,~\gamma\rightarrow0 
~~\text{and}~~\frac\delta\gamma\rightarrow0.
\end{equation*}

The next contribution of this paper is to develop a fully discrete scheme with thorough error analysis.
To numerically recover the initial data, we discretize the proposed regularization scheme 
by using piecewise linear finite element method (FEM) in space with spatial mesh size $h$, 
and  backward Euler convolution quadrature scheme (CQ-BE) in time with temporal step size $\tau$.
Then the numerical discretization introduces additional discretization error.
We establish an \textsl{a priori} error bound for the fully discrete scheme.
This estimate provides a useful guideline to choose suitable discretization parameters $h$ and $\tau$ and regularization parameter $\gamma$
according to the \textsl{a priori} known noise level $\delta$.
In particular, let $U_{h,\gamma}^{0,\delta}$ be the numerical reconstruction of initial condition.
Under those assumptions in the stability estimate \eqref{eqn:stab-intro},
we show that (Theorem \ref{thm:fully-err})
\begin{equation*}
\|U_{h,\gamma}^{0,\delta}-u_0\|_{L^2\II}\le c\Big(\gamma^\frac q2+ \delta\gamma^{-1} +
h^2\gamma^{-1}+\tau |\log \tau| (h^2\gamma^{-1}+1)  \Big),
\end{equation*}
if $\| u_0 \|_{\dH q}\le c$ with $q\in (0,2]$. Then with the choice
$\gamma \sim \delta^{\frac{2}{q+2}}$, $h\sim \delta^\frac 12$ and $\tau |\log \tau|  \sim \delta^{\frac{q}{q+2}}$,
we obtain the optimal approximation error of order 
$O(\delta^{\frac{q}{q+2}})$.
Moreover, for $u_0\in L^2\II$, there holds
\begin{equation*}
\| U_{h,\gamma}^{0,\delta} -  u_0\|_{L^2\II} \rightarrow 0\quad \text{as}~~\delta,\gamma,h,\tau\rightarrow0,
~~\frac\delta\gamma\rightarrow0~~\text{and}~~ \frac{h^2}{\gamma} \rightarrow0.
\end{equation*}
In this proof, we combine several useful tools, including the error analysis of the direct 
problem with respect to the problem data \cite{JinLiZhou:2019,JinLiZhou:2020nm},
the smoothing properties of discrete solution operators with frozen diffusion coefficients \cite{ZhangZhou:2020},
and the analysis of the conditional stability estimate \eqref{eqn:stab-intro} in the discrete sense.

The rest of the paper is organized as follows. In section \ref{sec:2} we provide some preliminary results 
about solution regularity, smoothing properties of solution operators and derive conditional stability of the inverse problem. 
In section \ref{sec:3} we discuss the regularization scheme by quasi-boundary value method. In section \ref{sec:4} we propose and analyze a fully discrete scheme for solving the backward problem.
Finally, in section \ref{sec:5} we present some numerical examples to illustrate and complete the theoretical analysis. 

Here we introduce some notations used throughout the paper. Under  conditions \eqref{Cond-1-t}--\eqref{Cond-2-t}, we define the abstract time-dependent elliptic
operator: 
$$
A(t)\phi=- \nabla\cdot(a(x,t)\nabla \phi)\quad \text{with} ~~\text{Dom}(A(t))=H^1_0(\Omega)\cap H^2(\Omega)
$$
 for all $t\in[0,T]$. By the complex
interpolation method \cite{Triebel:1978}, this implies
$$
\text{Dom}(A(t)^\gamma)=\dot H^{2\gamma}(\Omega)=(L^2(\Omega),H^1_0(\Omega)\cap H^2(\Omega))_{[\gamma]} ,\quad\forall\, t\in[0,T], \,\,\,\forall\,\gamma\in[0,1],
$$
Equivalently, it can be defined via spectral decomposition of
the operator $A(t)$ \cite[Chapter 3]{Thomee:2006}. Let $\{(\lambda_j,\varphi_j)\}_{j=1}^n$ be the eigenpairs of $A(t_*)$ for a fixed $t_* \in [0,T]$ with multiplicity
counted and $\{\varphi_j\}_{j=1}^\infty$ be an orthonormal basis in $L^2(\Omega)$. Then the Hilbert space $\dot{H}^{\gamma}
(\Omega)$ can be equivalently defined as
\begin{equation*}
  \dot H^{\gamma}(\Omega) = \Big\{v\in L^2(\Omega): \sum_{j=1}^\infty \lambda_j^{\frac {\gamma}{2}}(v,\varphi_j)^2<\infty\Big\}.
\end{equation*}
For $\gamma\in[0,2]$ we also denote by $\dot H^{-\gamma}(\Omega)$ the dual space of $\dot H^{\gamma}(\Omega)$. Then the norm
of $\dot H^{-\gamma}(\Omega)$ satisfies
\begin{equation*}
\|v\|_{\dot H^{-\gamma}(\Omega)}
\sim \|A(t)^{-\frac{\gamma}{2}}v\|_{L^2(\Omega)}\quad\forall\, v\in\dot H^{-\gamma}(\Omega),\,\,\forall\, t\in[0,T].
\end{equation*}

\section{Stability of the backward subdiffusion in Sobolev spaces} \label{sec:2}
First we recall basic properties of the subdiffusion model with a time-independent
diffusion coefficient, i.e., $a(x,t_*)$ for some $t_*\ge0$. Accordingly, consider the problem
\begin{equation}\label{PDE-independent}
\Dal u(t) + A(t_*)u(t) = f(t)\ \,\,\, \forall t\in(0,T],
\quad \mbox{with }u(0)=u_0.
\end{equation}
By means of Laplace transform,
the solution $u(t)$ can be represented by \cite[eq. (2.13)]{JinLiZhou:2019}
\begin{align}\label{eqn:Sol-expr-u-const}
u(t)= F(t;t_*)u_0 + \int_0^t E(t-s;t_*) f(s) \d s ,
\end{align}
where the solution operators $F(t;t_*)$ and $E(t;t_*)$ are defined by
\begin{equation}\label{eqn:sol-op-1}
 F(t;t_*) = \frac{1}{2\pi\mathrm{i}} \int_{\Gamma_{\theta,\kappa}} e^{zt} z^{\alpha-1} (z^\alpha
 + A(t_*))^{-1}  \,\d z,~~\text{and}~~
E(t;t_*) = \frac{1}{2\pi\mathrm{i}} \int_{\Gamma_{\theta,\kappa}} e^{zt} (z^\alpha +A(t_*))^{-1}\,\d z.
\end{equation}
with integration over a contour $\Gamma_{\theta,\kappa}\subset\mathbb{C}$ (oriented with an increasing imaginary part):
\begin{equation*}
  \Gamma_{\theta,\kappa}=\left\{z\in \mathbb{C}: |z|=\kappa, |\arg z|\le \theta\right\}\cup
  \{z\in \mathbb{C}: z=\rho e^{\pm\mathrm{i}\theta}, \rho\ge \kappa\} .
\end{equation*}
Throughout, we fix $\theta \in(\frac{\pi}{2},\pi)$ so that $z^{\al} \in \Sigma_{\al\theta}\subset
\Sigma_{\theta}:=\{0\neq z\in\mathbb{C}: {\rm arg}(z)\leq\theta\},$ for all $z\in\Sigma_{\theta}$.

The next lemma gives smoothing properties and asymptotics of  $F(t;t_*)$ and $E(t;t_*)$. The proof follows from
the resolvent estimate\cite[Example 3.7.5 and Theorem 3.7.11]{ArendtBattyHieber:2011}:
\begin{equation} \label{eqn:resol}
  \| (z+A)^{-1} \|\le c_\phi (|z|^{-1},\lambda^{-1})   \quad \forall z \in \Sigma_{\phi},
  \,\,\,\forall\,\phi\in(0,\pi) ,
\end{equation}
where $\|\cdot\|$ denotes the operator norm from $L^2(\Omega)$ to $L^2(\Omega)$,
and $\lambda$ denotes the smallest eigenvalue of $-\Delta$ with homogeneous Dirichlet boundary condition.
The proof of (i) and (ii) were given in \cite[Theorems 6.4 and 3.2]{Jin:2021book},
and (iii) were proved by Sakamoto  and Yamamoto in \cite[Theorem 4.1]{SakamotoYamamoto:2011}.

\begin{lemma}\label{lem:op}
Let $F(t;t_*)$ and $E(t;t_*)$ be the solution operators defined in \eqref{eqn:sol-op-1} for any $t_*\ge0$
Then they satisfy the following properties for all $t>0$
\begin{itemize}
\item[$\rm(i)$] $\|A(t_*) F (t;t_*)v\|_{L^2\II} +   
t^{1-(2-k)\alpha}  \| A(t_*)^{k} E (t;t_*)v  \|_{L^2\II} \le c  t^{-\alpha} \|v\|_{L^2\II}$ with $k=1,2$;
\item[$\rm(ii)$] $\|F(t;t_*)v\|_{L^2\II}+ t^{1-\alpha}\|E(t;t_*)v\|_{L^2\II} \le c \min(1,  t^{-\alpha}) \|v\|_{L^2\II}$;
\item[$\rm(iii)$] $\| F(t; t_*)^{-1} v \|_{L^2\II} \le c (1+ t^\alpha) \| v \|_{\dH2}$ for all $v\in \dH 2$.
\end{itemize}
The constants in all above estimates are uniform in $t$, but they are only dependent of $t_*$ and $T$.
\end{lemma}

Next, we turn to the subdiffusion with a time-dependent coefficient.
The overall proof strategy is to employ a perturbation argument \cite{JinLiZhou:2019}.
and then to properly resolve the singularity. Specifically,
for any fixed $t_*\in(0,T]$, we rewrite problem \eqref{eqn:fde-t} into
\begin{align}\label{eqn:reform}
\left\{\begin{aligned}\Dal u(t) + A(t_*)u(t) &=  (A(t_*)-A(t))u(t) + f(t) ,\quad \forall t\in(0,T],\\
        u(0)&=u_0.
\end{aligned}\right.
\end{align}
By \eqref{eqn:Sol-expr-u-const}, the solution $u(t)$ of \eqref{eqn:reform} is given by
\begin{align}\label{eqn:Sol-expr-u}
u(t)&= F(t;t_*)u_0+ \int_0^t E(t-s;t_*)(f(s) + (A(t_*)-A(s))u(s))\d s.
\end{align}

The following perturbation estimate  will be used extensively.
See similar results in {\cite[Corollary 3.1]{JinLiZhou:2019}}.
\begin{lemma}\label{lem:conti-A}
Under conditions \eqref{Cond-1-t}--\eqref{Cond-2-t}, there holds that 
\begin{equation*}
\|(A(t)-A(s))v\|_{\dH p}\le c\min(1,|t-s|)\|v\|_{\dH{p+2}},~~p\in [-2,0].
\end{equation*}
\end{lemma}
\begin{proof}
The condition \eqref{Cond-2-t} implies the case that $p=0$. The case $p=-2$ 
is has been proved in  {\cite[Corollary 3.1]{JinLiZhou:2019}}. 
Then the intermediate case follows from the interpolation \cite[Section 2.5] {Lions:2011}.
\end{proof} 

Next, we state a few regularity results. The proof of these results can be found in, e.g.,
 \cite{Bajlekov:2001, SakamotoYamamoto:2011, JinLiZhou:2019}.
\begin{theorem} \label{thm:reg-u}
Let $u(t)$ be the solution to \eqref{eqn:fde-t}. Then the following statements hold.
\begin{itemize}
  \item[$\rm(i)$] If $u_0 \in \dH q$ with $s\in[0,2]$ and $f=0$, then there holds
\begin{equation*}
\|  \partial_t^{(m)}  u(t) \|_{\dH p} \le c t^{\frac{(s-p)\alpha}{2}-m} \| u_0 \|_{\dH q}
\end{equation*}
with $0\le p-q\le 2$ and  $m = 0,1$. The constant $c$ in the estimate depends on $T$ and $\alpha$.
  \item[$\rm(ii)$]If $u_0=0$ and $f\in L^p(0,T;L^2(\Omega))$ with $1<p<\infty$, then there holds
\begin{equation*}
\|  u\|_{L^p(0,T;\dot H^2(\Omega))}
+\|\Dal u\|_{L^p(0,T;L^2(\Omega))}
\le c\|f\|_{L^p(0,T;L^2(\Omega))}.
\end{equation*}
Moreover, if $f\in L^p(0,T;L^2(\Omega))$ with $1/\alpha<p<\infty$, then $u(t)$ is the solution to problem \eqref{eqn:fde-t}
such that $u\in C([0,T];L^2\II)$.   The constant $c$ in the estimate depends on $T$ and $\alpha$.
\end{itemize}
\end{theorem}
The next lemma provides an a priori estimate similar to Theorem \ref{thm:reg-u} (i).
Note that the generic constant in the new estimate is independent of $T$.
\begin{lemma}\label{lemma:prior-estimate}
Suppose that $u_0 \in L^2\II$ and $f=0$.
Let $u(t)$ be the solution to the subdiffusion problem \eqref{eqn:fde-t}.
Under conditions \eqref{Cond-1-t}--\eqref{Cond-2-t}, there holds
\begin{align}\label{eqn:est-u-1}
 &\| u(t) \|_{L^2\II} \le c \min(1,t^{-\alpha}) \| u_0 \|_{L^2\II}~~\text{and}~~ \| u(t) \|_{\dH2} \le c \,e^{ct} t^{-\alpha} \| u_0 \|_{L^2\II}
  \quad \text{for all}~~ t>0
\end{align}
Meanwhile, for any $\epsilon\in(0,1/\alpha-1)$ and $t>0$, there holds that 
\begin{equation}\label{eqn:est-u-2}
  \| u(t)\|_{\dH 2} \le c  t ^{-(1-\epsilon)\alpha}\|u_0\|_{L^2\II}.
\end{equation}
All the positive constants $c$ in above estimates are independent of $t$ and $T$.
\end{lemma}
\begin{proof}
We define an operator $\underline{A}= -c_0 \Delta$.
Then by condition \ref{Cond-1-t}, the operator $A(t) - \underline{A}$
is selfadjoint and positive semidefinite for all $t\ge 0$. Then we rewrite the equation  \eqref{eqn:fde-t} as
$$  \partial_t^\alpha u(t) +  \underline{A} u(t) =  (\underline{A}-A(t)) u(t)  \quad \text{for all}~~ t\in(0,\infty).  $$
Taking inner product with $u(t)$ on the above equation and integrating by parts, we obtain
$$  (\partial_t^\alpha u(t), u(t)) +  c_0 \| \nabla u(t) \|_{L^2\II}^2
=  \big((c_0-a(\cdot,t)) \nabla u(t), \nabla u(t)\big)\le 0 \quad \text{for all}~~ t\in(0,\infty).  $$
Using the facts that $(\partial_t^\alpha u(t), u(t)) \ge \| u(t) \|_{L^2\II} \partial_t^\alpha \| u(t) \|_{L^2\II}$ \cite[Lemma 6.1(iii)]{Jin:2021book}
and Poincar\'e inequality we arrive at
$$   \partial_t^\alpha \| u(t)\|_{L^2\II}  +  c  \| u(t) \|_{L^2\II}   \le 0 \quad \text{for all}~~ t\in(0,\infty),  $$
for some constant $c$ uniform in $t$. Then the comparison principle for fractional ODEs \cite[Theorem 2.3]{Lakshmikantham:2007} leads to
$$  \| u(t)\|_{L^2\II} \le E_{\alpha,1}(-c t^\alpha) \| u_0 \| \le \frac{c}{1+ct^\alpha} \| u_0 \|_{L^2\II}. $$
This immediately leads to the desired claim \eqref{eqn:est-u-1}.

Next, we apply the relation \eqref{eqn:Sol-expr-u}, Lemmas \ref{lem:op} and \ref{lem:conti-A} (with $p=2$) to obtain for any $t_*\in(0,T]$
\begin{align*}
  \| u(t_*)\|_{\dH2} &\le \| F(t_*;t_*) u_0\|_{\dH2} + c \int_0^{t_*} \|  A(t_*) E(t_*-s;t_*) \| \, \|  (A(t_*) -A(s))u(s) \|_{L^2\II}\,\d s  \\
  &\le c t_*^{-\alpha} \| u_0 \|_{L^2\II} + c \int_0^{t_*}  \|   u(s) \|_{\dH2}\,\d s.
\end{align*}
Then the Gronwall's inequality implies for any $t>0$
$$\| u(t) \|_{\dH2} \le c \,e^{ct} t^{-\alpha} \| u_0 \|_{L^2\II}.$$
Meanwhile, Lemma \ref{lem:conti-A} leads to the estimate for $\beta=(1+\epsilon)\alpha$ with $\epsilon\in(0,1/\alpha-1)$  
\begin{align*}
  \| u(t_*)\|_{\dH2} &\le \| F(t_*;t_*) u_0\|_{\dH2} + \int_0^{t_*} \| A(t_*)^2 E(t_*-s;t_*) \| \, \| I - A(t_*)^{-1} A(s) \| \,  \|  u(s) \|_{L^2\II}\,\d s  \\
  &\le c t_*^{-\alpha} \| u_0 \|_{L^2\II} +  \int_{0}^{t_*}(t_*-s)^{-1+\epsilon\alpha} s^{-\alpha} \,\d s
  \le c_\epsilon t_*^{-(1-\epsilon)\alpha}.
\end{align*}
for any $t_*>0$. This completes the proof of \eqref{eqn:est-u-2}.
\end{proof}
Using the superposition principle, we consider the homogeneous source condition, i.e., $f=0$, without loss of generality.
Then the corresponding backward subdiffusion problem reads: find $u(0)$ such that
\begin{align}\label{eqn:fde-back}
\partial_t^\alpha u + A(t) u = 0 \,\,\,\forall\, t\in(0,T] \quad \text{with} \quad u(T)&= g  \qquad\mbox{in}\,\,\,\Omega.
\end{align}

The next theorem provides a stability estimate for the backward problem of \eqref{eqn:fde-back} 
when $T$ is sufficiently small.
\begin{theorem}\label{thm:stab-small}
Suppose that $u_0 \in L^2\II$ and $f=0$.
Let $u(t)$ be the solution to \eqref{eqn:fde-t}.
Under conditions \eqref{Cond-1-t}--\eqref{Cond-2-t},
there exists a positive constant $T_0$ such that for any $T\le T_0$ there holds
$$ \|  u_0  \|_{L^2\II} \le c(1+T^\alpha) \|  u(T) \|_{H^2\II}, $$
where the constant $c$ depends on $T_0$ and $T$.
\end{theorem}

\begin{proof}
We rearrange the terms in relation \eqref{eqn:Sol-expr-u} with $t_* = T$ to obtain
\begin{equation}\label{eqn:Sol-expr-u-1}
 u_0 = F(T;T)^{-1} \Big[u(T) - \int_0^T E(T-s;T) (A(T) - A(s)) u(s) \d s\Big].
 \end{equation}
Taking $L^2\II$ norm on both sides of the above relation, we apply Lemma \ref{lem:op} (iii) to obtain
\begin{align*}
\| u_0 \|_{L^2\II} 
 & \le  C (1+T^\alpha) \Big(\| u(T)  \|_{\dH2} + \int_0^{T}     \| A(T) E(T-s;T)\| \, \|(A(T)-A(s)))u(s) \|_{L^2\II}\,\d s \Big).
\end{align*}
According to Lemmas \ref{lem:conti-A} with $p=0$ and  \ref{lem:op} (i) we arrive at
\begin{align*}
\| u_0 \|_{L^2\II}   & \le  c(1+T^\alpha) \Big(\| u(T)  \|_{\dH2} + \int_0^{T}  \| u(s) \|_{\dH2}\,\d s \Big).
\end{align*}
Then this together with the estimate \eqref{eqn:est-u-1} implies
\begin{align*}
\| u_0 \|_{L^2\II}
&\le  c (1+T^\alpha) \Big(\| u(T)  \|_{\dH2} + \int_0^{T} e^{cs} s^{-\alpha} \,\d s \| u_0 \|_{L^2\II}\Big)\\
&\le c (1+T^\alpha) \Big(\| u(T)  \|_{\dH2} + ce^{cT}T^{1-\al} \| u_0 \|_{L^2\II}\Big)
\end{align*}
Let  be the constant that 
\begin{equation}
\label{cond:T0}
c (1+T_0^\alpha)e^{cT_0}T_0^{1-\al} < \frac12, ~~ T_0<1.
\end{equation}
 Then for any $T\le T_0$
\begin{align*}
\| u_0 \|_{L^2\II}  \le c (1+T^\alpha) \| u(T)  \|_{\dH2} .
\end{align*}
This completes the proof of the lemma.
\end{proof}
\vskip5pt

Next, we derive a stability estimate for a large $T$. To this end, we need the following assumption.
\begin{assumption}\label{ass:large-t}
There exists constants $c_2>0$ and $\kappa>0$ such that
\begin{align*}
 | \partial_t \nabla_x a(x,t)| +   | \partial_t a(x,t)|\le c_2 t^{-\kappa}\quad  \forall\, (x,t)\in \Omega\times(0,\infty).
\end{align*}
\end{assumption}

Under the condition, we have the following perturbation estimate. 
The proof is similar to that of Lemma \ref{lem:conti-A}. The proof is provided in Appendix A for completeness.
\begin{lemma}\label{lem:cond-AT}
Under Conditions \eqref{Cond-1-t}-\eqref{Cond-2-t} and Assumption \ref{ass:large-t},  there holds for all $t,s\ge1$
\begin{equation*}
\|(A(t)-A(s))v\|_{\dH p} \le c \min\big(1,\min(t,s)^{-\kappa}|t-s|\big)\|v\|_{\dH{p+2}},\qquad \forall~~ p\in[-2,0]  
\end{equation*}
\end{lemma}

The next theorem provides a stability result in case of sufficiently large T.
\begin{theorem}\label{thm:stab-big}
Suppose that $u_0 \in L^2\II$ and $f=0$. Let conditions \eqref{Cond-1-t}-\eqref{Cond-2-t} and Assumption \ref{ass:large-t} be valid.
Let $u(t)$ be the solution to the subdiffusion problem \eqref{eqn:fde-t}.
Then there exists  positive $T_1>1$ such that for any $T\ge T_1$ there holds
$$ \|  u_0  \|_{L^2\II} \le c  (1+T^\alpha) \|  u(T) \|_{H^2\II}, $$
where the constant $c$ depends on $T_1$ and $T$.
\end{theorem}
\begin{proof}
Using  \eqref{eqn:Sol-expr-u-1} and taking $L^2$ norm on both sides,  we apply again Lemma \ref{lem:op} (iii) to obtain
\begin{align*}
\| u_0 \|_{L^2\II}
 & \le  c (1+T^\alpha) \Big(\| u(T)  \|_{\dH2} + \Big\|\int_0^{T}     A(T) E(T-s;T)\big(  A(T) -A(s) \big) u(s) \,\d s \Big\|_{L^2\II}\Big).
\end{align*}
Applying Lemma \ref{lem:cond-AT} with $p=-2$, we have for sufficiently small $\epsilon>0$
$$\|(I-A(t)^{-1}A(s))v\|_{L^2\II} \le c \min(t,s)^{-\kappa\alpha(1+\epsilon)}|t-s|^{\alpha(1+\epsilon)}\|v\|_{L^2\II} $$
This together with Lemma \ref{lem:op} (i) and the \text{a priori} estimate \eqref{eqn:est-u-2},
imply
\begin{align*}
\|  A(T) E(T-s;T)\big(  A(T) -A(s) \big) u(s)   \|_{L^2\II}
&\le  \| A(T)^2 E(T-s;T)\| \, \| (I-A(T)^{-1}A(s)) u(s) \|_{L^2\II}\\
&\le c T^{-\kappa\alpha(1+\epsilon)}  (T-s)^{-1+\epsilon\alpha}  \, \|   u(s) \|_{L^2\II} \\
&\le c T^{-\kappa\alpha(1+\epsilon)}  (T-s)^{-1+\epsilon\alpha}  \, s^{-(1-\epsilon)\alpha}\|   u(0) \|_{L^2\II}
\end{align*}
for all $s\in[T/2,T]$. Then we arrive at 
\begin{align*}
\Big\|\int_{{T}/2}^{T}      A(T)  E(T-s;T)\big(  A(T) -A(s) \big) u(s) \,\d s \Big\|_{L^2\II} \le c T^{-\alpha-\kappa\alpha(1+\epsilon)+2\epsilon\alpha}  \| u_0 \|_{L^2\II} .
\end{align*}

Meanwhile, we apply Lemmas \ref{lem:conti-A} and \ref{lem:op} again to derive
$$\| A(T)^2 E(T-s;T)\| \, \| I - A(T)^{-1}A(s)  \| \le c (T-s)^{-1-\alpha}\quad \text{for all}~~ s\in(0,T/2].$$
This together with the estimate \eqref{eqn:est-u-1} leads to
\begin{align*}
 &\quad \int_0^{T/2}     \| A(T)^2 E(T-s;T)\| \, \| I - A(T)^{-1}A(s)  \| \, \| u(s) \|_{L^2\II}\,\d s  \\
 &\le c \int_0^{T/2} (T-s)^{-1-\alpha} s^{-(1-\epsilon)\alpha}   \,\d s \| u(0) \|_{L^2\II}
\le c T^{-(2-\epsilon)\alpha}\| u(0) \|_{\dH 2}.
\end{align*}
To sum up, we arrive at the estimate
\begin{align*}
\| u_0 \|_{L^2\II}
 & \le  c (1+T^\alpha) \| u(T)  \|_{\dH2} + c (1+T^\alpha) (T^{-\kappa\alpha(1+\epsilon)-\alpha+2\epsilon\alpha}  + T^{-(2-\epsilon)\alpha})\| u_0 \|_{L^2\II}
\end{align*}
Then choosing a sufficiently small $\epsilon$, there exists $T_1>1$ sufficiently large such that
\begin{equation}
\label{cond:T1}
c  (1+T_1^\alpha) (T_1^{-\kappa\alpha(1+\epsilon)-\alpha+2\epsilon\alpha}  + T_1^{-(2-\epsilon)\alpha}) = \frac 12
\end{equation}
and hence for any $T\ge T_1$, there holds the desired stability estimate.
\end{proof}

In Sections \ref{sec:3} and \ref{sec:4}, we shall discuss respectively the regularization and a fully discrete scheme
with rigorous numerical analysis. The stability estimate in Theorems \ref{thm:stab-small} and \ref{thm:stab-big}
provides a key tool in the coming numerical analysis. 
Therefore, from now on, we suppose the following assumption are valid.

\begin{assumption}\label{ass}
Suppose Conditions \eqref{Cond-1-t}--\eqref{Cond-2-t} and one of the following conditions are valid.
\begin{itemize}
\item[(i)] $T\le T_0$, where $T_0$ be a sufficiently small constant;
\item[(ii)] Assumption \ref{ass:large-t} holds and $T\ge T_1$ where $T_1$ be a sufficiently large constant. 
\end{itemize}
\end{assumption}

 
\section{Regularization and convergence analysis}\label{sec:3}
In practice, the observational data often suffers from noise, i.e., \eqref{eqn:noisy data}.
In this section, we study a simple regularization scheme by using the quasi boundary value method. 
Let $\tud(t) \in \dH1$ be the regularizing solution such that
\begin{equation}\label{eqn:fde:reg:noisy}
\begin{aligned}
\partial_t^\alpha \tud(t) + A(t) \tud(t) &= 0 ~~  \,\,\,\forall\, t\in(0,T] \quad \text{with}
\quad \text{with} \quad \gamma\tud(0) + \tud(T)&= g^\delta.
\end{aligned}
\end{equation}
where $\gamma$ denotes a positive regularization parameter. 
To derive an error estimate for  $\tud(0) - u(0)$, we introduce 
 an auxiliary function $\tu(t) \in \dH1$ satisfying
\begin{align}\label{eqn:fde-back-reg}
\partial_t^\alpha \tu(t) + A(t) \tu(t) = 0 \,\,\,\forall\, t\in(0,T] \quad \text{with} \quad \gamma\tu(0) + \tu(T)&= g.
\end{align}
Then using the solution representation
\begin{align*}
\tu(T) = F(T;T)\tu(0) + \int_0^T E(T-s;T) (A(T) - A(s)) \tu(s) \,\d s
\end{align*}
we have the relation
\begin{equation*} 
\gamma \tu(0) +  F(T;T)\tu(0) + \int_0^T E(T-s;T) (A(T) - A(s)) \tu(s) \,\d s = g.
\end{equation*}
Therefore, we derive
\begin{align}\label{eqn:reg-sol-0}
\begin{aligned}
 \tu(0) &=(\gamma I + F(T;T))^{-1} \Big[ g - \int_0^T E(T-s;T) (A(T) - A(s)) \tu(s) \,\d s  \Big].
\end{aligned}
\end{align}
Similarly, we have
\begin{align}\label{eqn:reg-sol-d}
\begin{aligned}
 \tud(0) &=(\gamma I + F(T;T))^{-1} \Big[ g^\delta - \int_0^T E(T-s;T) (A(T) - A(s)) \tud(s) \,\d s  \Big].
\end{aligned}
\end{align} 

 We begin with the following  lemma on solution operator with fixed-time operator $A(T)$.
 These estimates have been proved in \cite[Lemma 3.3]{ZhangZhou:2021} by means of spectral decomposition. 
\begin{lemma}
\label{lem:op-reg}
Let $0\le p\le q\le 2+p$. Then there holds the estimates for any $\gamma\in (0,1]$
\begin{equation*}
\|(\gamma I+F(T;T))^{-1}\|_{\dH p}\le c (1+T^\alpha)^{\frac{q-p}{2}} \gamma^{-(1+\frac{p-q}{2})}\|v\|_{\dH q}
\quad\text{and}\quad \|F(T;T)(\gamma I+F(T;T))^{-1}\|\le c.
\end{equation*}
All the constants are independent of $p$, $q$, $T$ and $\gamma$.
\end{lemma} 
Also we need the following regularity of the regularized solution.
\begin{lemma}\label{lem:reg-back}
Let $\tu(t)$ be the solution to \eqref{eqn:fde-back-reg}.
Suppose Conditions \eqref{Cond-1-t}--\eqref{Cond-2-t} and one of the following conditions are valid.
\begin{itemize}
\item[(i)] $T\le T_0$, where $T_0$ be a sufficiently small constant;
\item[(ii)] Assumption \ref{ass:large-t} holds and $T\ge T_1$ where $T_1$ be a sufficiently large constant. 
\end{itemize}
Then there holds for any $p\in [0,2]$, 
\begin{equation*}
\|\tu(0)\|_{\dH p} \le c \gamma^{-\frac p2}\|u_0\|_{L^2\II}
\end{equation*}
where the constant $c$ depends on $T_0$ and $T_1$.
\end{lemma}
\begin{proof}
By means of the representation \eqref{eqn:reg-sol-0}, Theorem \ref{thm:reg-u} and Lemma \ref{lem:op-reg},
\begin{align*}
\|\tu(0)\|_{L^2\II}&= \Big\|(\gamma I+F(T;T))^{-1}\Big(g - \int_0^T E(T-s;T)(A(T)-A(s))\tu(s)\d s\Big)\Big\|_{L^2\II}\\
&\le c(1+T^\al)\|g\|_{\dH2}+\|F(T;T)^{-1} \int_0^TE(T-s;T)(A(T)-A(s))\tu(s)\d s\|_{L^2\II}\\
&\le c_T\|u_0\|_{L^2\II}+\|F(T;T)^{-1} \int_0^TE(T-s;T)(A(T)-A(s))\tu(s)\d s\|_{L^2\II}
 \end{align*}
Then the desired result with $p=0$ 
follows immediately from the proof of theorems \ref{thm:stab-small} and \ref{thm:stab-big}.

Next, we turn to the case that  $p=2$. Similarly, we apply the representation \eqref{eqn:reg-sol-0}
and Lemma \ref{lem:op-reg} again to obtain  
\begin{align*}
\| \tu(0)\|_{\dH2}
&\le  c \Big\|A(T)(\gamma I+F(T;T))^{-1}\Big(g  - \int_0^T E(T-s;T)(A(T)-A(s))\tu(s)\d s\Big)\Big\|_{L^2\II}\\
 &\le c\gamma^{-1} \|g \|_{\dH 2} +\int_0^T \|F(T;T)^{-1}  A(T) E(T-s;T)(A(T)-A(s))\tu(s)\|_{L^2\II} \d s \\
  &\le c_T \gamma^{-1}  \| u_0 \|_{L^2\II} +c(1+T^\alpha)\int_0^T \|   A(T)^2 E(T-s;T)(A(T)-A(s))\tu(s)\|_{L^2\II} \d s.
\end{align*}
Using  Lemma \ref{lemma:prior-estimate} and Poincare inequality, we have
\begin{align}\label{eqn:est-tu-a}
\| \tu(t)\|_{\dH2} \le  ce^{ct} t^{-\alpha} \|  \tu(0) \|_{\dH2} 
\quad \text{and}\quad \| \tu(t)\|_{\dH 2} \le c  t ^{-(1-\epsilon)\alpha}\|\tu(0)\|_{\dH2},
\end{align}
with any small parameter $\epsilon>0$ and $t>0$,
and all the positive constants $c$ in above estimates are independent of $t$ and $T$.
Next, we repeat the argument in theorems \ref{thm:stab-small} and \ref{thm:stab-big}.
Now Lemmas \ref{lem:op} and \ref{lem:conti-A} (with $p=0$) imply that  
\begin{align*}
\|\tu(0)\|_{\dH 2}  &\le c\gamma^{-1} T^{-\alpha}\| u_0 \|_{L^2\II} 
+c(1+T^\alpha)\int_0^T \|   A(T)^2 E(T-s;T)(A(T)-A(s))\tu(s)\|_{L^2\II} \d s \\
&\le c\gamma^{-1} T^{-\alpha}\| u_0 \|_{L^2\II}   + c(1+T^\alpha) 
\int_0^T (T-s)^{-\alpha}\| \tu(s)\|_{\dH2} \d s\\
&\le c\gamma^{-1} T^{-\alpha}\| u_0 \|_{L^2\II}   + c(1+T^\alpha) 
\int_0^T (T-s)^{-\alpha} e^{cs}s^{-\alpha} \| \tu(0)\|_{\dH2} \d s
\end{align*}
We combine this and \eqref{eqn:est-tu-a} to arrive at
\begin{align*}
\|\tu(0)\|_{\dH 2}  \le c\gamma^{-1} T^{-\alpha}\| u_0 \|_{L^2\II}+ c(1+T^\al)T^{1-2\al}e^{cT}\|\tu(0)\|_{\dH 2}.
\end{align*}
Then by choosing  small $T_0$ such that $c(1+T_0^\al)T_0^{1-2\al}e^{cT_0}<\frac 12$,  we arrive at 
\begin{equation*}
\|\tu(0)\|_{\dH 2}\le c\gamma^{-1} T^{-\alpha }\|u_0\|_{L^2\II}\qquad \text{for all}~~T\in(0,T_0).
\end{equation*}
Next we consider the case that $T$ is sufficiently large, and we let Assumption \ref{ass:large-t} be valid.
Then we apply Lemma \ref{lem:cond-AT} with $p=0$ to arrive at
$$\|(A(t) - A(s))v\|_{L^2\II} \le c \min(t,s)^{-\kappa\alpha(1+\epsilon)}|t-s|^{\alpha(1+\epsilon)}\big)\|v\|_{\dH2} $$
for sufficiently small $\epsilon$.
This together with Lemma \ref{lem:op} and the estimate \eqref{eqn:est-tu-a}  lead to
\begin{align*}
\|  A(T)^2 E(T-s;T)\big(  A(T) -A(s) \big) \tu(s)   \|_{L^2\II}
&\le \|  A(T)^2 E(T-s;T)\|  \, \| \big(  A(T) -A(s) \big) \tu(s)   \|_{L^2\II}\\
&\le c T^{ -\kappa\alpha(1+\epsilon)}  (T-s)^{-1+\epsilon\alpha}  \, \|  \tu(s) \|_{\dH2} \\
&\le c T^{-\kappa\alpha(1+\epsilon)}  (T-s)^{-1+\epsilon\alpha}  \, s^{-(1-\epsilon)\al}\|   \tu(0) \|_{\dH 2}
\end{align*}
for all $s\in[T/2,T]$. Then we arrive at 
\begin{align*}
\int_{{T}/2}^{T}   \big\|   A(T)^2  E(T-s;T)\big(  A(T) -A(s) \big) \tu(s)\big\|_{L^2\II} \,\d s  
\le c T^{-\alpha-\kappa\alpha(1+\epsilon)+2\epsilon\alpha} \|   \tu(0) \|_{\dH 2} .
\end{align*}

Meanwhile, we apply Lemmas \ref{lem:conti-A} and \ref{lem:op} again to derive
$$\| A(T)^2 E(T-s;T) (A(t) - A(s))v  \| \le c (T-s)^{-1-\alpha} \|  v \|_{\dH2}  \quad \text{for all}~~ s\in(0,T/2].$$
This together with the estimate \eqref{eqn:est-tu-a} leads to
\begin{align*}
 &\quad \int_0^{T/2}    \big\|   A(T)^2  E(T-s;T)\big(  A(T) -A(s) \big) \tu(s)\big\|_{L^2\II} \,\d s  \\
 &\le c \int_0^{T/2} (T-s)^{-1-\alpha} s^{-(1-\epsilon)\alpha}   \,\d s \| \tu(0) \|_{\dH2}
\le c T^{-(2-\epsilon)\alpha}\| u(0) \|_{\dH 2}.
\end{align*}
To sum up, we arrive at the estimate
\begin{align*}
\| u_0 \|_{L^2\II}
 & \le  c\gamma^{-1} T^{-\alpha}\| u_0 \|_{L^2\II} 
 + c (1+T^\alpha) (T^{-\kappa\alpha(1+\epsilon)-\alpha+2\epsilon\alpha}  + T^{-(2-\epsilon)\alpha})\| u_0 \|_{L^2\II}
\end{align*}
Then choosing a sufficiently small $\epsilon$, there exists $T_1>1$ sufficiently large such that
\begin{equation}
\label{cond:T1}
c  (1+T_1^\alpha) (T_1^{-\kappa\alpha(1+\epsilon)-\alpha+2\epsilon\alpha}  + T_1^{-(2-\epsilon)\alpha}) = \frac 12
\end{equation}
and hence for any $T\ge T_1$, there holds the desired stability estimate for $p=2$.
 \end{proof}
 
 The following lemma is about the estimate of the regularization with the backward solution.
\begin{lemma}\label{lem:err-reg0}
Let $u$ and $\tu$ be the solutions to the backward problem \eqref{eqn:fde-back} 
and regularized problem \eqref{eqn:fde-back-reg}, respectively. 
Suppose Assumption \ref{ass} is valid. Then if $u_0\in \dH q$ with $q\in (0,2]$ there holds 
\begin{equation*}
 \| \tu(0) - u(0)  \|_{L^2\II} \le c \gamma^{\frac q2} \|  u_0 \|_{\dH q}.
 \end{equation*}
where the constant $c$ depends on $T_0$ and $T_1$. Moreover, for $u_0\in L^2\II$, there holds
\begin{equation*}
\lim_{\gamma\rightarrow0^+}\|{u}_\gamma(0) -  u(0)\|_{L^2\II} =0 .
\end{equation*}
\end{lemma}
\begin{proof}
We let $ e:= \tu - u$, it would satisfy 
\begin{equation*}
\Dal  e+A(t)  e = 0,~~ \gamma  e(0) +  e(T) = -\gamma u_0,
\end{equation*}  
which further implies
\begin{equation}\label{repre:tilde-e}
  e(0) =(\gamma I + F(T;T))^{-1} \Big[-\gamma u(0) -     \int_0^T E(T-s;T) (A(T) - A(s))  e(s) \,\d s  \Big]
\end{equation}
 lemma \ref{lem:op-reg} implies its estimate that 
\begin{align*}
\|  e(0)\|_{L^2\II}&\le c\gamma^\frac{q}{2}\|u_0\|_{\dH q} + \|(\gamma I+F(T;T))^{-1}\int_0^T E(T-s;T)(A(T)-A(s))  e(s)\d s\|_{L^2\II}\\
&\le c\gamma^\frac{q}{2}\|u_0\|_{\dH q} + \|F(T;T)^{-1}\int_0^T E(T-s;T)(A(T)-A(s))  e(s)\d s\|_{L^2\II}.
\end{align*}
Then the desired result 
follows immediately from the proof of theorems \ref{thm:stab-small} and \ref{thm:stab-big}.

Next, we consider the case that $u_0 \in L^2\II$. 
For an arbitrary $\tilde u_0 \in \dH2$, let 
%
 $\tilde u(t)$ and $\tilde u_\gamma(t)$ be the functions  respectively satisfying 
\begin{align*} 
\partial_t^\alpha \tilde u(t) + A(t) \tilde u(t) = 0 \,\,\,\forall\, t\in(0,T] \quad \text{with} \quad  \tilde u (0)  = \tilde u_0,
\end{align*}
and
\begin{align*} 
\partial_t^\alpha \tilde u_\gamma(t) + A(t) \tilde u_\gamma(t) = 0 \,\,\,\forall\, t\in(0,T] \quad \text{with} \quad \gamma\tilde u_\gamma(0) + \tilde u_\gamma(T)&= \tilde u(T).
\end{align*}
We have proved that
\begin{equation*}
 \| \tilde u_\gamma(0) - \tilde u(0)  \|_{L^2\II} \le c \gamma \|  \tilde u_0 \|_{\dH 2}.
 \end{equation*}
 Meanwhile, using the argument in theorems \ref{thm:stab-small} and \ref{thm:stab-big}, we have
$$ \| \tilde u_\gamma(0) - u_\gamma(0) \|_{L^2\II} \le c \|  u_0 - \tilde u_0  \|_{L^2\II} \le c\epsilon. $$
 As a result, we apply triangle inequality to obtain
\begin{align*}  
\| u_\gamma(0) - u_0 \|_{L^2\II} &\le \|  u_0 - \tilde u_0  \|_{L^2\II} + \|  u_\gamma(0)  - \tilde u_\gamma(0)  \|_{L^2\II}
+ \|   \tilde u_\gamma(0) -  \tilde  u_0 \|_{L^2\II} \\
&\le c  \|  u_0 - \tilde u_0  \|_{L^2\II} + c \gamma \|  \tilde u_0 \|_{\dH 2}.
\end{align*}
Let $\epsilon$ be an arbitrarily small number.
Using the density of $\dH2$ in $L^2\II$, we choose $\tilde u_0$
such that $c  \|  u_0 - \tilde u_0  \|_{L^2\II} \le \frac{\epsilon}{2}$.
Moreover, let $\gamma_0$ be the constant that $c \gamma_0 \|  \tilde u_0 \|_{\dH 2} < \frac{\epsilon}{2}$.
Therefore, for all $\gamma\le \gamma_0$, we have
$ \|{u}_\gamma(0) -  u(0)\|_{L^2\II} \le \epsilon$. Then the proof is complete.
 \end{proof}

Then we are ready to state our main theorem to show the error for the regularizing solution $\tud(0)$.
 \begin{theorem}\label{thm:err-reg}
 Let $u$ and $\tud$ be the solutions to the backward problem \eqref{eqn:fde-back} 
and regularized problem \eqref{eqn:fde:reg:noisy}, respectively. 
Suppose Assumption \ref{ass} is valid.
 Then if $\| u_0 \|_{\dH q}\le c$ with $q\in (0,2]$ there holds 
\begin{equation*}
 \| \tud(0) - u(0)  \|_{L^2\II} \le c \,\Big(\delta \gamma^{-1} +  \gamma^{\frac q2}\Big).
\end{equation*}
Moreover, for $u_0\in L^2\II$, there holds
\begin{equation*}
\| \tud(0) -  u(0)\|_{L^2\II} \rightarrow 0\quad \text{as}~~\delta,~\gamma\rightarrow0 
~~\text{and}~~\frac\delta\gamma\rightarrow0.
\end{equation*}
\end{theorem}

\begin{proof}
To show the error estimate, we consider the splitting
$$ \tud(t) - u(t) =  (\tud(t) - \tu(t)) +  (\tu(t) - u(t)) = \vartheta(t) + \varrho(t).$$
Using the solution representation \eqref{eqn:reg-sol-0} and \eqref{eqn:reg-sol-d}, we have
\begin{align*}
\| \vartheta(0) \|_{L^2\II}&\le \|(\gamma I + F(T;T))^{-1} (g - g^\delta)\|_{L^2\II}   \\
 &\quad +\|(\gamma I + F(T;T))^{-1}\int_0^T E(T-s;T) (A(T) - A(s)) \theta(s) \,\d s \|_{L^2\II}.
\end{align*}
Using Lemma \eqref{lem:op-reg}, we derive
\begin{align*}
\| \vartheta(0) \|_{L^2\II}&\le c \gamma^{-1} \| g - g^\delta \|_{L^2\II} 
+\| F(T;T)^{-1}\int_0^T E(T-s;T) (A(T) - A(s)) \theta(s) \,\d s \|_{L^2\II}\\
&\le c \gamma^{-1} \delta
+\| F(T;T)^{-1}\int_0^T E(T-s;T) (A(T) - A(s)) \theta(s) \,\d s \|_{L^2\II}.
\end{align*}
Applying the argument in theorems \ref{thm:stab-small} and \ref{thm:stab-big}, we 
conclude that $ \| \vartheta(0) \|_{L^2\II} \le c \gamma^{-1} \delta $. This estimate and
Lemma \ref{lem:err-reg0} lead to the desired result.
\end{proof}

\section{Fully discretization scheme and error analysis}\label{sec:4} 
In this section, we shall propose and analyze a completely discrete scheme for solving the 
backward problem. To begin with, we study the semidiscrete scheme using the finite element methods. 
The semidiscrete solution plays an important role in the analysis of completely discrete scheme.

\subsection{Semidiscrete scheme for solving the problem}
To begin with, we study the semidiscrete scheme using the finite element methods. 
Let  ${\{\mathcal{T}_h\}}_{0<h<1}$ be a  family
of shape regular and quasi-uniform partitions of the domain $\Omega$ into $d$-simplexes, called
finite elements, with $h$ denoting the maximum diameter of the elements.
We consider the finite element space $X_h$ defined by
\begin{equation}\label{eqn:fem-space}
  X_h =\left\{\chi\in C(\bar\Omega)\cap H_0^1: \ \chi|_{K}\in P_1(K),
 \,\,\,\,\forall K \in \mathcal{T}_h\right\}
\end{equation}
where $P_1(K)$ denotes the space of linear polynomials on $K$.
Then we define the $L^2(\Omega)$ projection $P_h:L^2(\Omega)\to X_h$, by
\begin{align*}
(P_h \psi,\chi) & =(\psi,\chi) \quad\forall \chi\in X_h,\psi\in L^2(\Omega).
\end{align*}

Then $P_h$  satisfies the following approximation properties \cite[Chapter 1]{Thomee:2006}
\begin{align}
\label{estimate:P_h}
\|P_hv-v\|_{\L2Om}+h\|\nabla(P_hv-v)\|_{\L2Om}\le ch^q\|v\|_{H^q(\Omega)},\ \forall v\in \dH q,\ q=1,2.
\end{align}
The semidiscrete standard Galerkin FEM of problem \eqref{eqn:fde-t} reads: find $u_h\in X_h$ such that 
\begin{equation}\label{eqn:fde-t-h}
(\Dal u_h(t),\chi)+(a(\cdot,t)\nabla u_h(t),\nabla \chi)=(f(\cdot,t),\chi),~~\forall \chi\in X_h,~~t\in(0,T],
\text{ with }u_h(0) = P_hu_0.
\end{equation}

We also need a time-dependent discrete elliptic operator $A_h(t):X_h\to X_h$ by 
\begin{equation*}
(A_h(t)v_h,\chi) = (a(\cdot,t)\nabla v_h,\nabla\chi),~~\forall v_h,\chi\in X_h.
\end{equation*}
With conditions \eqref{Cond-1-t}-\eqref{Cond-2-t}, $A_h(t)$ is bounded and invertible on $X_h$, and problem \eqref{eqn:fde-t-h} can be written as 
\begin{equation}
\label{eqn:fde-t-h-Ah}
\Dal u_h+A_hu_h = P_hf,~~\forall t\in(0,T],~~u_h(0) = P_hu_0.
\end{equation}
Besides, we have the following perturbation result, which has been proved in \cite[Remark 3.1]{JinLiZhou:2019}.
\begin{lemma}
\label{lem:cond-Ah}
Under condition \eqref{Cond-1-t}-\eqref{Cond-2-t}, there holds
\begin{align*}
\|(I-A_h(t)^{-1}A_h(s) v_h\|_{L^2\II}\le c\min(1,|t-s|)\| v_h\|_{L^2\II}
\end{align*}
\end{lemma}

Next, we introduce a time-dependent Ritz projection operator $R_h(t):H_0^1\II\to X_h$:
\begin{equation}\label{eqn:Rh}
(a(\cdot,t)\nabla R_h(t)\fy,\nabla \chi)=(a(\cdot,t)\nabla\fy,\nabla \chi),~~\forall \fy\in H_0^1\II,~~\chi\in X_h.
\end{equation}
It is well-known that the Ritz projection satisfies the following approximation property 
\cite[p.99]{LuskinRannacher:1982}:
\begin{equation}
\label{estimate:Rh}
\|R_h(t)v-v\|_{L^2\II}+h\|\nabla(R_h(t)v-v)\|_{L^2\II}\le ch^q\|v\|_{H^q\II},~~\forall v \in \dH q,~~q=1,2.
\end{equation}

Next, with
Assumption \ref{ass:large-t}, we have an updated version of the discrete perturbation estimate.
\begin{lemma}\label{lem:cond-AhT}
With conditions \eqref{Cond-1-t}-\eqref{Cond-2-t} and Assumption \ref{ass:large-t}, we have 
for all $v_h \in X_h$
\begin{align*}
\|(I-A_h(t)^{-1}A_h(s) v_h\|_{L^2\II}\le c\min(1,\min(t,s)^{-\kappa}|t-s|)\|v_h\|_{L^2\II},~~\forall t,s>1.
\end{align*}
\end{lemma}
\begin{proof}
Let $w_h = A_h(t)^{-1}A_h(s)v_h$. Then we have $A_h(t)w_h = A_h(s)v_h$ and hence
\begin{equation*}
(a(\cdot,t)\nabla w_h,\nabla\chi_h) = (a(\cdot,s)\nabla v_h,\nabla \chi_h),~~\forall \chi_h\in X_h.
\end{equation*}
This further implies the relation
\begin{equation*}
(a(\cdot,t)\nabla (v_h-w_h),\nabla \chi_h) = ((a(\cdot,t)-a(\cdot,s))\nabla v_h,\nabla \chi_h),~~\forall \chi\in X_h.
\end{equation*}
Let $\phi$ be the weak solution to the following elliptic problem:
\begin{equation*}
(a(\cdot,t)\nabla \phi,\nabla \chi) = ((a(\cdot,t)-a(\cdot,s))\nabla v_h,\nabla \chi),~~\forall \chi\in \dH 1.
\end{equation*}
Then  Lax-Milgram lemma and Assumption  \ref{ass:large-t} implies the following \textsl{a priori} estimate 
\begin{equation*}
\|\phi\|_{\dH 1}\le c\|(a(\cdot,t)-a(\cdot,s))\nabla v_h\|_{L^2\II}\le c\min(1,\min(t,s)^{-\kappa}|t-s|)\|v_h\|_{\dH 1},
\end{equation*}
Using the fact that $w_h - v_h = R_h(t) \phi$, the approximation property \eqref{estimate:Rh}, and the inverse inequality, we derive
\begin{align*}
\|w_h-v_h-\phi\|_{L^2\II}&\le ch\|\fy\|_{\dH 1}\le ch\min(1,\min(t,s)^{-\kappa}|t-s|)\|v_h\|_{\dH 1} \\
&\le c\min(1,\min(t,s)^{-\kappa}|t-s|)\|v_h\|_{L^2\II}.
\end{align*}
According triangle inequality we have
\begin{align*}
\|w_h-v_h\|_{L^2\II}\le c\min(1,\min(t,s)^{-\kappa}|t-s|)\|v_h\|_{L^2\II}+\|\phi\|_{L^2\II}.
\end{align*}
Next, we apply the duality argument to derive a bound for $\|\phi\|_{L^2\II}$. Let $\xi \in \dH2$ be 
the function such that  $ A(t) \xi = \phi$. Then
\begin{align*}
\|  \phi \|_{L^2\II}^2 &= |(a(\cdot,t)\nabla\phi,\nabla\xi)|
 = |((a(\cdot,t)-a(\cdot,s))\nabla v_h,\nabla\xi)|\\
&\le  |(v_h,(a(\cdot,t)-a(\cdot,s)) \Delta\xi)| +  |(v_h,\nabla (a(\cdot,t)-a(\cdot,s)) \cdot\nabla \xi)|\\
 &\le c\min(1,\min(t,s)^{-\kappa}|t-s|) \|v_h\|_{L^2\II}\|\xi\|_{\dH 2}\\
&\le c\min(1,\min(t,s)^{-\kappa}|t-s|) \|v_h\|_{L^2\II}\|\phi\|_{L^2\II}.
\end{align*}
This completes the proof of the lemma.
\end{proof}
Next we derive some semidiscrete solution representation analogue to  \eqref{eqn:Sol-expr-u}, that is given any $t_*\in(0,T]$,
\begin{equation}
\label{eqn:Sol-expr-uh}
u_h(t) = F_h(t;t_*)u_h(0)+\int_0^t E_h(t-s;t_*)(P_hf(s)+(A_h(t_*)-A_h(s))u_h(s))ds
\end{equation} 
where the solution operators $F_h(t;t_*)$ and $E_h(t;t_*)$ can be written as 
\begin{equation}\label{eqn:op-h}
 F_h(t;t_*) = \frac{1}{2\pi\mathrm{i}} \int_{\Gamma_{\theta,\kappa}} e^{zt} z^{\alpha-1} (z^\alpha
 + A_h(t_*))^{-1}  \,\d z,~~\text{and}~~
E_h(t;t_*) = \frac{1}{2\pi\mathrm{i}} \int_{\Gamma_{\theta,\kappa}} e^{zt} (z^\alpha +A_h(t_*))^{-1}\,\d z.
\end{equation}
For any fixed $t_*$, the discrete operators $F_h(t;t_*)$ and $E_h(t;t_*)$ 
satisfy the following smoothing property, whose proof is identical to that of Lemma \ref{lem:op}.
\begin{lemma}\label{lem:op-d}
Let $F_h(t;t_*)$ and $E_h(t;t_*)$ be the discrete solution operators 
defined in \eqref{eqn:op-h} for any $t_*\in[0,T]$.
Then they satisfy the following properties for all $t>0$ and $v_h \in X_h$
\begin{itemize}
\item[$\rm(i)$] $\|A_h(t_*) F_h (t;t_*)v_h\|_{L^2\II} + t^{1-(2-k)\alpha} \| A_h(t_*)^k  E_h (t;t_*)v_h  \|_{L^2\II}
 \le c  t^{-\alpha} \|v_h\|_{L^2\II}$, with $k=1,2$;
\item[$\rm(ii)$] $\|F_h(t;t_*)v_h\|_{L^2\II}+ t^{1-\alpha}\|E_h(t;t_*)v_h\|_{L^2\II} \le c \min(1,  t^{-\alpha}) \|v_h\|_{L^2\II}$;
\item[$\rm(iii)$] $\|F_h(t; t_*)^{-1} v_h \|_{L^2\II} \le c (1+ t^\alpha) \| A_h(t_*) v_h \|_{L^2\II}$.
\end{itemize}
The constants in all above estimates are uniform in $t$, but they are only dependent of $t_*$ and $T$.
\end{lemma}

Analogue to Lemma \ref{lem:op-reg}, we have the following result.
\begin{lemma}\label{lem:op-reg-semi}
Let $F_h(t;t_*)$ be the discrete solution operator defined in \eqref{eqn:op-h}. 
For all $0<t\le T$, $t_*\in(0,T]$ and $v_h\in X_h$, we have
\begin{equation*}
\|   (\gamma I+F_h(T;T))^{-1}v_h \|  \le c \gamma^{-1}\| v_h \|_{L^2\II}~~\text{and}
\quad \|F_h(T;T)(\gamma I+F_h(T;T))^{-1}v_h\|_{L^2\II}\le c\|v_h\|_{L^2\II},
\end{equation*}
where the constant $c$ is independent of $t$, $\gamma$ and $h$.
\end{lemma}

The following Lemma provides an error estimate for the semidiscrete error of the direct problem, see \cite[Theorem 3.2]{JinLiZhou:2019} for a detailed proof.
\begin{lemma}
\label{lem:semi-direct}
Let $u$ and $u_h$ be the solutions to \eqref{eqn:fde-t} and \eqref{eqn:fde-t-h} respectively. If $u_0\in L^2$ and $f\equiv 0$, then there holds that 
\begin{equation*}
\|(u_h-u)(t)\|_{L^2\II}\le ch^2 t^{-\al}\|u_0\|_{L^2\II}~~\text{for all}~~t\in(0,T],
\end{equation*}
where the constant $c$ is independent of $t$ and $h$.
\end{lemma}
After proposing many results about solving direct problem, we shall propose a semidiscrete scheme for solving the backward problem.

We apply the regularized semidiscrete scheme: find $u_{\gamma,h}(t)\in X_h$ such that 
\begin{equation}
\label{eqn:fde-h-reg}
\Dal \tuh(t)+A_h(t)\tuh(t) = 0,~~ 0<t\le T,\qquad \gamma \tuh(t) + \tuh(T) = P_hg.
\end{equation}
Then analogue to \eqref{eqn:reg-sol-0} we have 
\begin{equation}
\label{eqn:repre-tudh}
\begin{aligned}
\tuh(0)&=(\gamma I+F_h(T;T))^{-1} \left[P_hg-\int_0^TE_h(T-s;T)(A_h(T)-A_h(s))\tuh(s)ds\right].
\end{aligned}
\end{equation}
Next we shall derive a preliminary estimate for the proof of the semidiscrete error $\tuh-\tu$.
\begin{lemma}
\label{lem:semi-preliminary}
Let $\tu(t)$ be the solution to the backward regularized problem \eqref{eqn:fde-h-reg}. Then fix any $t_*\in (0,T]$ there holds that 
\begin{equation*}
\|\int_0^{t_*} E_h(t_*-s;t_*)A_h(s)(R_h(s)-P_h)\tu(s)ds\|_{L^2\II}\le ch^2 {\max\{ t_*^{-\alpha},t_*^{1-\alpha}\}}\|\tu(0)\|_{L^2}.
\end{equation*}
The constant $c$ is independent of $t$ and $t_*$.
\end{lemma}
\begin{proof}
Let $\fy_h$ be the solution to the following semidiscrete problem
\begin{equation*}
\Dal \fy_h(t)+A_h(t)\fy_h(t) =0,\qquad \fy_h(0) = P_h \tu(0).
\end{equation*}
{Lemma \ref{lem:semi-direct}} implies that 
\begin{equation*} 
\|(\fy_h-\tu)(t)\|_{L^2\II}\le ch^2 t^{-\alpha} \|\tu(0)\|_{L^2\II}.
\end{equation*}
Then we consider the splitting
\begin{equation*}
(\fy_h-\tu)(t)=(\fy_h-P_h\tu)(t)+(P_h\tu-\tu)(t):=\zeta_h(t)+\rho(t).
\end{equation*}
The approximation property \eqref{estimate:P_h} and the regularity estimate in 
Theorem \ref{thm:reg-u} give that 
\begin{equation*} 
\|\rho(t)\|_{L^2\II}\le ch^2\|\tu(t)\|_{{\dH2}}\le ch^2 t^{-\alpha} \|\tu(0)\|_{L^2\II}.
\end{equation*}
Then by triangle's inequality, we obtain
\begin{equation}\label{eqn:zh-1}
\|\zeta_h(t)\|_{L^2\II}\le \| \rho(t) \|_{L^2\II} + \| (\fy_h-\tu)(t) \|_{L^2\II} \le ch^2 t^{-\alpha} \|\tu(0)\|_{L^2\II}.
\end{equation} 
Meanwhile notice that 
\begin{equation*}
\Dal \zeta_h(t)+A_h(t)\zeta_h(t)=A_h(t)(R_h(t)-P_h)\tu(t),~~T\ge t>0,\qquad \zeta(0) = 0.
\end{equation*}
Then for any $t_* \in (0,T]$, $\zeta_h(t_*)$ could be written as 
\begin{equation*}
\zeta_h(t_*) = \int_0^{t_*} E_h(t_*-s;t_*)A_h(s)(R_h(s)-P_h)\tu(s)ds+\int_0^{t_*} E_h(t_*-s;t_*)(A_h(t_*)-A_h(s))\zeta_h(s)ds
\end{equation*}
we apply Lemmas \ref{lem:cond-AhT} and \ref{lem:op-d}, and the estimate \eqref{eqn:zh-1} to derive
\begin{equation*}
\begin{aligned}
&\quad\|\int_0^t E_h(t-s;t_*)A_h(s)(R_h(s)-P_h)\tu(s)ds\|_{L^2\II} \\
&\le \|\zeta_h(t)\|_{L^2\II} 
+ c\int_0^{t_*} \|\zeta_h(s)\|_{L^2\II} \,\d s
\le c (t_*^{-\alpha}+t_*^{1-\alpha})h^2\|\tu(0)\|_{L^2\II} \\
& \le {c h^2 \max\{t_*^{-\alpha},t_*^{1-\alpha}\}\|\tu(0)\|_{L^2\II}}.
\end{aligned}
\end{equation*}
\end{proof}

Next, we state a key  lemma providing an estimate for the discretization error $\tuh-\tu$.
\begin{lemma}
\label{lem:semi-err}
Let $\tu(t)$, $\tuh(t)$ be the solutions to problem \eqref{eqn:fde-back-reg} and \eqref{eqn:fde-h-reg} respectively. 
Suppose Assumption \ref{ass} is valid. 
Then there holds  
$$ \| \tuh(0) - \tu(0)  \| \le c h^2\gamma^{-1}\|u_0\|_{L^2\II}, $$
where the constant $c$ is independent on $\gamma$, $h$ and $t$.
\end{lemma}
\begin{proof}
We use the splitting 
\begin{equation*}
(\tuh-\tu)(0) = (\tuh-P_h\tu)(0)+(P_h\tu-\tu)(0):=\zeta_h(0)+\rho(0).
\end{equation*}
From the approximation property \eqref{estimate:P_h} and Lemma  \ref{lem:reg-back}, we obtain
\begin{equation*}
\|\rho(0)\|_{L^2\II}\le ch^2\|\tu(0)\|_{\dH 2}\le ch^2 \gamma^{-1}\|u_0\|_{L^2\II}.
\end{equation*} 

Now we turn to the bound of $\zeta_h(t)$. 
Using the fact  $ A_h(t)R_h(t)v = P_hA(t)v$, 
we observe that
\begin{equation}
\label{eqn:zeta}
\Dal \zeta_h(t)+ A_h(t)\zeta_h(t) = A_h(t)(R_h(t)-P_h)\tu(t)~~ \text{for}~~t\in(0,T],
\quad\text{with}~~\gamma \zeta_h(0)+\zeta_h(T) = 0.
\end{equation}
For any  $t_*\in (0,T]$, we have the solution representation from \eqref{eqn:Sol-expr-uh} that 
\begin{align*}
\zeta_h(t) &= F_h(t;t_*)\zeta_h(0)+\int_0^t E_h(t-s;t_*)A_h(s)(R_h(s)-P_h)\tu(s)\, \d s \\
&\qquad +\int_0^t E_h(t-s;t_*)(A_h(t_*)-A_h(s))\zeta_h(s)\,\d s.
\end{align*}
Then with $t = t_* = T$ we apply $\gamma \zeta_h(0)+\zeta_h(T) = 0$ to derive
\begin{equation*}
\begin{aligned}
\zeta_h(0) &= (\gamma I +F_h(T;T))^{-1}\int_0^TE_h(T-s;T)(A_h(s)(P_h-R_h(s))\tu(s)\,\d s \\
&\qquad -(\gamma I+F_h(T;T))^{-1}\int_0^T E_h(T-s;T)(A_h(T)-A_h(s))\zeta(s)\,\d s. 
\end{aligned}
\end{equation*}
Now we apply Lemmas \ref{lem:op-reg-semi} and   \ref{lem:semi-preliminary} to obtain
\begin{align*}
\|\zeta_h(0)\|_{L^2\II}&\le c\gamma^{-1}\int_0^T \|E_h(T-s;T)A_h(s)(P_h-R_h(s))\tu(s)\|_{L^2\II}\d s\\
&+\|F_h(T;T)^{-1}\int_0^TE_h(T-s;T)(A_h(T)-A_h(s))\zeta(s) \d s\|_{L^2\II}\\
&\le {c_T}h^2\gamma^{-1} \|\tu(0)\|_{L^2\II} +c(1+T^\al)\int_0^T \|A_h(T;T)E_h(T-s;T)(A_h(T)-A_h(s))\zeta_h(s)\|_{L^2\II} \d s
\end{align*}
Next, we
 split $\zeta_h(s)$ into homogeneous part and inhomogeneous part. Let $\zeta_h(t) := \zeta_1(t)+\zeta_2(t)$ where 
\begin{align*}
\Dal \zeta_1(t)+A_h(t)\zeta_1(t) &= 0~~ \text{for}~~t\in(0,T],
\quad\text{with}~~\zeta_1(0) = \zeta_h(0),\\
\Dal \zeta_2(t)+A_h(t)\zeta_2(t) &= A_h(t)(R_h(t)-P_h)\tu(t)~~ \text{for}~~t\in(0,T],
\quad\text{with}~~\zeta_2(0) = 0,
\end{align*}
First of all, we fixed $t_*\in (0,T]$ and apply the solution representation in \eqref{eqn:Sol-expr-uh} 
and Lemmas \ref{lem:op-d}, \ref{lem:semi-preliminary} and \ref{lem:cond-Ah}, and hence derive
\begin{align*}
\|\zeta_2(t_*)\|_{L^2\II}&\le \int_0^{t_*} \|E_h(t_*-s;t_*)A_h(s)(R_h(s)-P_h)\tu(s)\|_{L^2\II}\d s\\
&+ \int_0^{t_*} \|E_h(t_*-s;t_*)(A_h(t_*)-A_h(s))\zeta_2(s)\|_{L^2\II}\d s\\
&\le ct_*^{-\al}h^2\|\tu(0)\|_{L^2\II} + \int_0^{t_*} \|\zeta_2(s)\|_{L^2\II}\d s.
\end{align*}
Then  Gronwall's inequality leads to
\begin{equation}\label{estimate:zeta_2}
\|\zeta_2(t)\|_{L^2\II}\le ch^2e^{ct}t^{-\al}\|\tu(0)\|_{L^2\II}.
\end{equation}
For $\zeta_1(t)$, we apply the similar argument in Lemma \ref{lemma:prior-estimate} to obtain
\begin{align*}
 &\| \zeta_1(t) \|_{L^2\II} \le c \min(1,t^{-\alpha}) \|  {\zeta_h(0)} \|_{L^2\II}~~\text{and}~~ \|A_h(T) \zeta_1(t)\|_{L^2\II} \le c \,e^{ct} t^{-\alpha} \| {\zeta_h(0)} \|_{L^2\II}
  \quad \text{for all}~~ t>0.
\end{align*}
All the positive constants $c$ in above estimates are independent of $t$ and $T$.
As a result, we have 
\begin{align*}
\|\zeta_h(0)\|_{L^2\II}&\le ch^2\gamma^{-1} \|\tu(0)\|_{L^2\II} +\sum_{i=1}^2c(1+T^\al)\int_0^T \|A_h(T;T)E_h(T-s;T)(A_h(T)-A_h(s))\zeta_i(s) \d s\|_{L^2\II}\\
&\le ch^2\gamma^{-1}\|\tu(0)\|_{L^2\II}+c(1+T^\al)\int_0^T \|A_h(T)E_h(T-s;T)(A_h(T)-A_h(s))\zeta_1(s)\|_{L^2\II} \d s.
\end{align*}
Applying the argument in theorems \ref{thm:stab-small} and \ref{thm:stab-big}, we 
conclude that $ \|\zeta_h(0)\|_{L^2\II} \le c h^2 \gamma^{-1} \|\tu(0)\|_{L^2\II} $. 
This completes the proof of the lemma.
\end{proof}
\subsection{Fully discrete scheme and error analysis}
To begin with, we introduce the fully discrete scheme for the direct problem.
We divide the time interval $[0,T]$ into a uniform grid, with $ t_n=n\tau$, $n=0,\ldots,N$, and $\tau=T/N$ being
the time step size. Then we approximate the fractional derivative by using 
the backward Euler convolution quadrature (with $\varphi^j=\varphi(t_j)$) \cite{Lubich:1986, JinLiZhou:2017sisc}:
\begin{equation*}
\bar\partial_\tau^\alpha \varphi^n = \sum_{j=0}^n \omega_{n-j}^{(\alpha)} (\varphi^{j} - \varphi^0),
\quad\mbox{ with }  ~ \omega_j^{(\alpha)}=  (-1)^j\frac{\Gamma(\alpha+1)}{\Gamma(\alpha-j+1)\Gamma(j+1)}.
\end{equation*}

The fully discrete scheme for problem \eqref{eqn:fde-t-h-Ah} reads: find ${U_h^n}\in X_h$ such that
\begin{equation}\label{eqn:fully}
\bDal U_h^n +A_h(t_n)U_h^n= P_h f(t_n),\quad n=1,2,\ldots,N,\quad \text{with}~~U_h^0 = P_hu_0,
\end{equation}  
By means of Laplace transform 
and perturbation argument, with $1\le n_*\le N$, 
the fully discrete solution $U^n$ can be written as \cite{JinLiZhou:2019, ZhangZhou:2020}
\begin{equation}
\label{eqn:repre-U}
U_h^n = F_{h, \tau}^n(n_*) U_h^0+ \tau \sum_{k=1}^n E_{h, \tau}^{n-k}(n_*)P_hf(t_k)+\tau\sum_{k=1}^nE_{h, \tau}^{n-k}(n_*)(A_h(t_{n_*})-A_h(t_k)) U_h^k
\end{equation}
with $n=1,2,\cdots,N$. Here the fully discrete operators $F_{h, \tau}^n(n_*)$ and $E_{h, \tau}^n(n_*)$ are defined by 
\begin{equation}\label{eqn:FE_ht-0}
\begin{aligned}
F_{h, \tau}^n(n_*) &= \frac{1}{2\pi\mathrm{i}}\int_{\Gamma_{\theta,\sigma}^\tau } e^{zt_n} {e^{-z\tau}} \delta_\tau(e^{-z\tau})^{\alpha-1}({ \delta_\tau(e^{-z\tau})^\alpha}+A_h(t_{n_*}))^{-1}\,\d z,\\
E_{h, \tau}^n(n_*) &= \frac{1}{2\pi\mathrm{i}}\int_{\Gamma_{\theta,\sigma}^\tau } e^{zt_n} ({ \delta_\tau(e^{-z\tau})^\alpha}+A_h(t_{n_*}))^{-1}\,\d z ,
\end{aligned}
\end{equation}
with $\delta_\tau(\xi)=(1-\xi)/\tau$ and the contour
$\Gamma_{\theta,\sigma}^\tau :=\{ z\in \Gamma_{\theta,\sigma}:|\Im(z)|\le {\pi}/{\tau} \}$ where $\theta\in(\pi/2,\pi)$ is close to $\pi/2$.
(oriented with an increasing imaginary part).
The next lemma gives elementary properties of the kernel $\delta_\tau(e^{-z\tau})$. The detailed proof has been given in
\cite[Lemma B.1]{JinLiZhou:2017sisc}.
\begin{lemma}\label{lem:delta}
For a fixed $\theta'\in(\pi/2,\pi/\alpha)$, there exists $\theta\in(\pi/2,\pi)$ \
positive constants $c,c_1,c_2$ $($independent of $\tau$) such that for all $z\in \Gamma_{\theta,\sigma}^\tau$
\begin{equation*}
\begin{aligned}
& c_1|z|\leq
|\delta_\tau(e^{-z\tau})|\leq c_2|z|,\qquad
\delta_\tau(e^{-z\tau})\in \Sigma_{\theta'}. \\
& |\delta_\tau(e^{-z\tau})-z|\le c\tau |z|^{2},\qquad
 |\delta_\tau(e^{-z\tau})^\alpha-z^\alpha|\leq c\tau |z|^{1+\alpha}.
 \end{aligned}
\end{equation*}
\end{lemma}

The next lemma provides some approximation properties of solution operators 
$F_{h,\tau}^n(n_*)$ and $E_{h,\tau}^n(n_*)$.
See \cite[Lemma 4.2]{ZZZ:2022} for the proof of the first estimate,
and  \cite[Lemma 4.5]{JinLiZhou:2019} for the second estimate.
\begin{lemma}
\label{lem:op-err:fully}
For the operator $F_h^\tau$ and $E_h^\tau$ defined in \eqref{eqn:FE_ht-0}, we have 
\begin{align*}
&\| A_h(t_{n_*})^\beta (F_{h,\tau}^n(n_*)-  F_{h}(t_n; t_{n_*}))\|_{L^2\II}\le c\tau t_n^{-1-\beta\al},\\
&\Big\|\tau A_h^\beta E_{h,\tau}^{{n_*}-k}(n_*)-\int_{t_{k-1}}^{t_k} A_h^\beta(t_{n_*}) E_h(t_{n_*}-s;t_{n_*}))ds\Big\|\le c\tau^2(t_{n_*}-t_k+\tau)^{-(2-(1-\beta)\al)}
\end{align*}
for any $\beta\in[0,1]$.
\end{lemma}

Note that the solution operators $F_{h,\tau}^n(n_*)$ and $E_{h,\tau}^n(n_*)$ satisfy the following smoothing properties, whose proof is identical to that of Lemma \ref{lem:op}.
See also a similar result in \cite[Lemma 4.3]{ZZZ:2022}.
\begin{lemma}
\label{lem:op:fully}
Let $F_h^\tau(n;n)$ and $E_h^\tau(n;n)$ be the operators defined in \eqref{eqn:FE_ht-0}. 
Then they satisfy the following properties for any $n\ge 1$ and $v_h\in X_h$,
\begin{itemize}
\item[$\rm(i)$] $\|
A_h(t_*) F_{h,\tau}^n (n_*)v_h\|_{L^2\II} + t_{n+1}^{1-(2-k)\alpha}  \| A_h(t_*)^k  E_{h,\tau}^n (n_*)v_h  \|_{L^2\II} \le c  t_{n+1}^{-\alpha} \|v_h\|_{L^2\II},~~ k = 1,2$;
\item[$\rm(ii)$] $\|F_{h,\tau}^n (n_*)v_h\|_{L^2\II}+ t_n^{1-\alpha}\|E_{h,\tau}^n (n_*)v_h\|_{L^2\II} \le c \min(1,  t_n^{-\alpha}) \|v_h\|_{L^2\II}$;
\item[$\rm(iii)$] $\| F_{h,\tau}^n (n_*)^{-1} v_h \|_{L^2\II} \le c (1+ t_n^\alpha) \| A_h(t_*) v _h\|_{L^2\II}$ 
\end{itemize}
 \end{lemma}

Next we introduce some \textsl{a priori} estimate for the discrete solution $U^n$ in \eqref{eqn:repre-U}, analogue to Lemma \ref{lemma:prior-estimate} for the continuous problem.
We provide the proof in Appendix B for completeness.
\begin{lemma}\label{lem:ful:a-priori}
Let $U^n$ be the solution to \eqref{eqn:fully}, then we have the following a-priori estimate ($f\equiv 0$)
\begin{equation*}
\|U_h^n\|_{L^2\II}\le c\min(1,t_n^{-\al})\|U_h^0\|_{L^2\II}
~~\text{and}~~\|A_h(T)U_h^n\|_{L^2\II}\le ce^{ct_n}t_n^{-\al}\|U_h^0\|_{L^2\II}~~\text{for} ~~n\ge1.
\end{equation*}
Moreover, for any $\epsilon\in(0,1/\alpha-1)$, there holds  
\begin{equation*}
\|A_h(T)U_h^n\|_{L^2\II} \le c t_n ^{-(1-\epsilon)\alpha}\|U_h^0\|_{L^2\II}~~\text{for all} ~~n\ge1.
\end{equation*}
All the constants in above estimates are independent of $h$, $n$, $N$, $\tau$ and $T$.
\end{lemma}

Now we introduce the fully discrete scheme for solving the backward problem: find 
$U_{h,\gamma}^{n,\delta} \in X_h$ for $n=1,\ldots,N$ such that 
\begin{equation}\label{eqn:back-fully}
\bDal U_{h,\gamma}^{n,\delta} +A_h(t_n)U_{h,\gamma}^{n,\delta} = P_hf(t_n) ~~\text{for all} ~~1\le n\le N,
\quad \text{with}~~ 
\gamma U_{h,\gamma}^{0,\delta} + U_{h,\gamma}^{N,\delta} = P_h g_\delta.
\end{equation}
Then $U_{h,\gamma}^{n,\delta}$ can be written as 
\begin{equation}\label{eqn:repre-Udga}
\begin{aligned}
U_{h,\gamma}^{0,\delta} &= (\gamma I+F_{h,\tau}^N(N))^{-1}\Big[P_hg_\delta-\tau \sum_{k=1}^N F_{h,\tau}^{N-k}(N)(A_h(T)-A_h(t_k))U_{h,\gamma}^{k,\delta}\Big].
\end{aligned}
\end{equation} 
The following lemma provides a useful estimate of the discrete  operator $(\gamma I+F_{h,\tau}^N(N))^{-1}$; see a detailed proof in \cite[Lemma 4.4]{ZhangZhou:2021}.
\begin{lemma}\label{lem:op-reg-fully}
Let $F_h^\tau(n;n_*)$ and $E_h^\tau(n;n_*)$ be the operators defined in \eqref{eqn:FE_ht-0}.
Then there holds 
\begin{equation*}
\|(\gamma I+F_h^\tau(N;N))^{-1}v_h\|_{L^2\II}\le c\gamma^{-1}\|v_h\|_{L^2\II}~~ \text{and}~~
\|F_h^\tau(N;N)(\gamma I+F_h^\tau(N;N))^{-1} v_h\|_{L^ 2\II}\le c
\end{equation*}
where $c$ is uniform in $T$, $h$, $\tau$ and $\gamma$.
\end{lemma}

To show the error between $U_{h,\gamma}^{N,\delta}$ and $u_0$, we introduce an
auxiliary function  $\bar U_{h,\gamma}^n  \in X_h $ such that
\begin{equation}\label{eqn:fde:ful}
\bDal \bar U_{h,\gamma}^n +A_h(t_n)\bar U_{h,\gamma}^n = P_hf(t_n) 
~~\text{for all} ~~1\le n\le N,
\quad \text{with}~~ \bar U_{h,\gamma}^n = \tuh(0),
\end{equation}
Then we have the following error estimate for the direct problem, according to \cite[Theorem 4.1]{JinLiZhou:2019}.
\begin{lemma} \label{lem:err:ful:forward}
Let $\tuh(t)$ and $\bar U_{h,\gamma}^n$ be the solution to \eqref{eqn:fde-t-h-Ah} and \eqref{eqn:fde:ful} with $f\equiv 0$, then we have 
\begin{equation*}
\| A_h(0) (\bar U_{h,\gamma}^n-\tuh(t_n))\|_{L^2\II} \le  c \tau \log (n+1) \max(t_n^{-\al-1}, t_n^{-\al}  ) \|\tuh(0)\|_{L^2\II}.
\end{equation*}
\end{lemma}
\begin{proof}
Let $e_n = \bar U_{h,\gamma}^n-\tuh(t_n)$. First of all, we recall \cite[Theorem 4.1]{JinLiZhou:2019} that
\begin{equation}\label{eqn:en41}
\| e_n \|_{L^2\II} \le  c\tau t_n^{-1}\log(n+1) \|\tuh(0)\|_{L^2\II}.
\end{equation}
We then use the solution representation \eqref{eqn:repre-U} to obtain 
\begin{equation*}
\bar U_{h,\gamma}^n = F_{h,\tau}^{n}(n_*) \tuh(0) 
+ \tau \sum_{k=1}^n E_{h,\tau}^{n-k}(n_*)(A_h(t_{n_*})-A_h(t_k)) \bar U_{h,\gamma}^k.
\end{equation*}
Then by means of \eqref{eqn:Sol-expr-uh}, we have for fixed $t_{n_*}$
\begin{align*}
&\quad A_h(t_{n_*})e_{n_*} = A_h(t_{n_*})(F_{h,\tau}^n(n_*)-F_h(t_n;t_{n_*}))\tuh(0) \\
&+ \tau\sum_{k=1}^{n_*} A_h(t_{n_*})E_{h,\tau}^{n_*-k}(n_*)(A_h(t_{n_*})-A_h(t_k)) \bar U_{h,\gamma}^k - \int_0^{t_{n_*}}A_h(t_{n_*}) 
E_h(t_{n_*}-s;t_{n_*})(A_h(t_{n_*})-A_h(s))\tuh(s) \, \d s\\
&+\tau\sum_{k=1}^{n_*}A_h(t_{n_*})E_{h,\tau}^{n_*-k}(n_*)(A_h(t_{n_*})-A_h(t_k))(\bar U_{h,\gamma}^k-\tuh(t_k))
 = I_1+I_2+I_3.
\end{align*}
 Lemma \ref{lem:op-err:fully} immediately implies the bound for $I_1$: 
$$\|I_1\|_{L^2\II}\le c\tau t_{n_*}^{-1-\al} \| \tuh(0)\|_{L^2\II}.$$
A slightly modification of \cite[Lemma 4.4]{JinLiZhou:2019} leads to a bound for $I_2$. In particular, we observe
\begin{align*}
I_2 &= \sum_{k=1}^{n_*} \Big[\tau A_h(t_{n_*})E_{h,\tau}^{n_*-k}(n_*)(A_h(t_{n_*})-A_h(t_k))- \int_{t_{k-1}}^{t_{k}}A_h(t_{n_*}) 
E_h(t_{n_*}-s;t_{n_*})(A_h(t_{n_*})-A_h(s))\d s\Big]\bar U_{h,\gamma}^k\\
& + \sum_{k=1}^{n_*} \int_{t_{k-1}}^{t_k}A_h(t_{n_*}) 
E_h(t_{n_*}-s;t_{n_*})(A_h(t_{n_*})-A_h(s)) \, \d s\,  e_k\\
& + \sum_{k=1}^{n_*} \int_{t_{k-1}}^{t_k}A_h(t_{n_*}) 
E_h(t_{n_*}-s;t_{n_*})(A_h(t_{n_*})-A_h(s))(\tuh(t_k)-\tuh(s)) \, \d s:=  I_{2,1}+I_{2,2}+I_{2,3},
\end{align*}
For $I_{2,1}$, by means of Lemma \ref{lem:op-err:fully} with $\beta =1$, \ref{lem:cond-Ah} and \ref{lem:op:fully} (i) with the solution representation \eqref{eqn:repre-U}, we arrive at 
\begin{align*}
\|I_{2,1}\|_{L^2\II}&\le \sum_{k=1}^{n_*} \|\Big[\tau A_h(t_{n_*})E_{h,\tau}^{n_*-k}(n_*)(I-A_h(t_k)A_h(t_{n_*})^{-1})\\
&\qquad - \int_{t_{k-1}}^{t_{k}}A_h(t_{n_*}) 
E_h(t_{n_*}-s;t_{n_*})(I-A_h(s)A_h(t_{n_*})^{-1})\d s\Big]\|\,\|A_h(t_{n_*})\bar U_{h,\gamma}^k\|_{L^2\II}\\
&\le c\sum_{k=1}^{n_*} \tau^2(t_{n_*}-t_k+\tau)^{-1} t_k^{-\al} \|\tuh(0)\|_{L^2\II} \\
&\le c \tau \log(n_*+1)t_n^{-\alpha} \|\tuh(0)\|_{L^2\II}.
\end{align*}
For $I_{2,2}$ we apply Lemmas \ref{lem:op-d}  (i) with $k=2$, Lemma \ref{lem:cond-Ah} and \textsl{a priori} estimate \eqref{eqn:en41} to derive
\begin{align*}
\|I_{2,2}\|_{L^2\II}&\le  c\tau t_{n_*}^{-\al-1}\log(n_*+1)\|\tuh(0)\|_{L^2\II}.
\end{align*}
Last, for the erm $I_{2,3}$, we denote 
\begin{equation*}
Q_k = \int_{t_{k-1}}^{t_k}A_h(t_{n_*}) 
E_h(t_{n_*}-s;t_{n_*})(A_h(t_{n_*})-A_h(s))(\tuh(t_k)-\tuh(s)).
\end{equation*} 
For $k=1$, we apply Lemmas \ref{lem:op-d} and \ref{lem:cond-Ah} to derive the bound
\begin{align*}
\|Q_1\|_{L^2\II}&\le \|\int_0^\tau A_h(t_{n_*}) 
E_h(t_{n_*}-s;t_{n_*})(A_h(t_{n_*})-A_h(s))\tuh(\tau)\,\d s\|_{L^2\II}\\
&\qquad +\|\int_0^\tau A_h(t_{n_*}) 
E_h(t_{n_*}-s;t_{n_*})(A_h(t_{n_*})-A_h(s))\tuh(s)\,\d s\|_{L^2\II}\\
& \le c\int_0^\tau (t_{n_*}-s)^{-\al}\d s\|\tuh(0)\|_{L^2\II} \le c\tau t_{n_*}^{-\al} \|\tuh(0)\|_{L^2\II}.
\end{align*}
Meanwhile, for $k\ge 2$, there holds that 
\begin{align*}
Q_k =  \int_{t_{k-1}}^{t_k}  A_h(t_{n_*}) 
E_h(t_{n_*}-s;t_{n_*})(A_h(t_{n_*})-A_h(s))\int_{s}^{t_k} \tuh'(\xi)\d \xi\,\d s.
\end{align*}
The discrete analogue to Theorem \ref{thm:reg-u}(i) (see detail proof in \cite[Theorem 2.3(i)]{JinLiZhou:2019}), $\tuh'(t)$ can be bounded by 
\begin{equation}
\label{eqn:est:tuh'}
\|\tuh'(t)\|_{L^2\II}\le ct^{-1}\|\tuh(0)\|_{L^2\II}.
\end{equation}
 Then by Lemmas \ref{lem:op-d}, \ref{lem:cond-Ah} and regularity estimate \eqref{eqn:est:tuh'} there holds 
\begin{align*}
\|Q_k\|_{L^2\II}&\le c \int_{t_{k-1}}^{t_k} \|A_h(t_{n_*}) 
E_h(t_{n_*}-s;t_{n_*})(A_h(t_{n_*})-A_h(s))\|\,\int_{s}^{t_k} \xi^{-1}\d \xi\,\d s\|\tuh(0)\|_{L^2\II}\\
&\le c\int_{t_{k-1}}^{t_k} (t_{n_*}-s)^{-\al}\int_s^{t_k} \xi^{-1}\d \xi\,\d s \|\tuh(0)\|_{L^2\II}\\
&\le c \tau \int_{t_{k-1}}^{t_k} (t_{n_*}-s)^{-\al}s^{-1}\d s \|\tuh(0)\|_{L^2\II}.
\end{align*}
Summing those terms from $k=2$ to $k=n_*$, we obtain
\begin{align*}
\sum_{k=2}^{n_*} \|Q_k\|_{L^2\II}&\le c\tau \|\tuh(0)\|_{L^2\II} \int_\tau^{t_{n_*}}(t_{n_*}-s)^{-\al}s^{-1}\d s\\
&\le  c\tau t_{n_*}^{-\al-1}\log(n_*+1)\|\tuh(0)\|_{L^2\II}.
\end{align*}
As a result, we arrive at
\begin{equation*}
\|I_2\|_{L^2\II} \le c\tau t_{n_*}^{-\alpha-1}  \log (n_*+1) \|\tuh(0)\|_{L^2\II}.
\end{equation*}
Finally, Lemmas \ref{lem:op:fully}, \ref{lem:cond-Ah} and the estimate \eqref{eqn:en41} imply that 
\begin{align*}
\|I_3\|_{L^2\II}&\le c \tau \sum_{k=1}^{n_*} \|A_h(t_{n_*})^2 
E_{h,\tau}^{{n_*}-k}(n_*)\|\,\|I-A_h(t_k)A_h(t_{n_*})^{-1}\|\, \| e_k\|_{L^2\II}\\
&\le c\tau \sum_{k=1}^{n_*}(t_{n_*} - t_k)^{-\alpha}\| e_k\|_{L^2\II} \\
&\le c \tau^2 \sum_{k=1}^{n_*}(t_{n_*} - t_k)^{-\alpha}t_k^{-1}\log(k+1) \|\tuh(0)\|_{L^2\II}\\
&\le c\tau \log(n_*+1) t_n^{-\alpha}\|\tuh(0)\|_{L^2\II}.
\end{align*}
This completes the proof of the lemma.
\end{proof}
Next, we introduce an auxiliary function
\begin{equation}\label{eqn:full-back-g}
\bDal U_{h,\gamma}^n +A_h(t_n)U_{h,\gamma}^n = 0
~~\text{for all} ~~1\le n\le N,
\quad \text{with}~~ 
\gamma U_{h,\gamma}^0+U_{h,\gamma}^N = P_hg.
\end{equation}
Then $U_{h,\gamma}^{0}$ can be written as 
\begin{equation}\label{eqn:repre-g}
\begin{aligned}
U_{h,\gamma}^{0} &= (\gamma I+F_{h,\tau}^N(N))^{-1}\Big[P_hg-\tau \sum_{k=1}^N F_{h,\tau}^{N-k}(N)(A_h(T)-A_h(t_k))U_{h,\gamma}^{k}\Big].
\end{aligned}
\end{equation}

Then the next lemma provides an estimate for $U_{h,\gamma}^{0,\delta}-U_{h,\gamma}^{0}$.
\begin{lemma}
\label{lem:err-ful-noise}
Let $U_{h,\gamma}^{n,\delta}$ and $U_{h,\gamma}^{n}$ be the solution to problems 
\eqref{eqn:back-fully} and \eqref{eqn:full-back-g} respectively.
Suppose Assumption \ref{ass} is valid. 
Then there holds 
\begin{equation*}
\|U_{h,\gamma}^{0,\delta}-U_{h,\gamma}^{0}\|_{L^2\II}\le c\delta\gamma^{-1},
\end{equation*}
where the constant $c$ is independent on $\gamma$, $h$, $\tau$ and $t$.
\end{lemma}
\begin{proof}
Let $e_n = U_{h,\gamma}^{n,\delta} - U_{h,\gamma}^{n}$.
Then $e_n$ satisfies the relation that
\begin{equation}\label{eqn:en-1}
\bDal e_n +A_h(t_n) e_n  =0 ~~\text{for all} ~~1\le n\le N,
\quad \text{with}~~ 
 \gamma  e_0+ e_N = P_h(g_\delta -g)
\end{equation}
Using representations \eqref{eqn:repre-Udga} and \eqref{eqn:repre-g} we obtain 
\begin{equation*}
\begin{aligned}
 e_0&=(\gamma I+F_{h,\tau}^{N}(N))^{-1} \Big[P_h(g_\delta-g) - \tau \sum_{k=1}^N E_{h,\tau}^{N-k}(N)(A_h(T)-A_h(t_k))  e_k\Big].
\end{aligned}
\end{equation*}
Now we apply Lemmas \ref{lem:op:fully} and \ref{lem:ful:a-priori} to obtain
\begin{equation*}
\begin{aligned}
&\| e_0\|_{L^2\II}\le c\delta\gamma^{-1} + \|F_h^\tau(N;N)^{-1}\tau\sum_{k=1}^NE_h^\tau(N-k;N)(A_h(T)-A_h(t_k)) e_k\|_{L^2\II}\\
&\le c\delta\gamma^{-1} + c(1+T^\al)\sum_{k=1}^N\|\tau A_h(T)E_h^\tau(N-t_k;N)(A_h(T)-A_h(t_k)) e_k\|_{L^2\II}.
\end{aligned}
\end{equation*}
Then the desired results follows immediately from the a priori estimate in Lemma \ref{lem:ful:a-priori}
and the same argument in theorems \ref{thm:stab-small} and \ref{thm:stab-big}.
\end{proof}

Time discretization would give the following fully error estimate.
\begin{lemma}\label{lem:err-reg:ful} 
Let $\tuh(t)$ and $U_{h,\gamma}^{n}$ be the solutions to \eqref{eqn:fde-h-reg} and \eqref{eqn:fde:ful} respectively.
Suppose Assumption \ref{ass} is valid. 
Then there holds 
$$ \|\tuh(0)- U_{h,\gamma}^{0} \| \le c \tau |\log\tau| (h^2\gamma^{-1}+1)  \| u_0 \|_{L^2\II},$$
where the constant $c$ is independent on $\gamma$, $h$ and $t$.
\end{lemma}
\begin{proof}
Let $\bar U_{h,\gamma}^n$ be the solution to \eqref{eqn:fde:ful} 
and $e_n = \bar U_{h,\gamma}^n- U_{h,\gamma}^n$, which satisfies the following equation 
\begin{equation}
\label{eqn:ful:err}
\bDal e_n  +A_h(t_n)e_n =0 ~~\text{for all} ~~1\le n\le N,
\quad \text{with}~~ 
\gamma e_0 + e_N =  \bar U_{h,\gamma}^N -  \tuh(T)=:Q.
\end{equation}
Then we apply the  representation of fully discrete scheme to derive
\begin{equation}
\label{eqn:repre:err:ful}
e_0 = (\gamma I+F_{h,\tau}^N(N)))^{-1}\Big[Q-\sum_{k=1}^N \tau E_{h,\tau}^{N-k}(N)(A_h(T)-A_h(t_k))e_k\Big].
\end{equation}
Lemmas \ref{lem:op:fully} and  \ref{lem:op-reg-fully} give that 
\begin{align*}
\|e_0\|_{L^2\II}
&\le  \Big\| F_{h,\tau}^N(N)^{-1}
\Big[Q-\sum_{k=1}^N \tau \E_h^\tau (N-k;N)(A_h(T)-A_h(t_k))e_k\Big]\Big\|_{L^2\II}\\
&\le c_T\|A_h(T)Q\|_{L^2\II} + c(1+T^\al)\|\sum_{k=1}^N\tau A_h(T) E_{h,\tau}^{N-k}(N) (A_h(T)-A_h(t_k))e_k\|_{L^2\II}.
\end{align*}
This combined with Lemma \ref{lem:err:ful:forward} leads to
 \begin{equation*}
\|e_0\|_{L^2\II} \le c_T \tau |\log\tau| \|\tuh(0)\|_{L^2\II}
+ c(1+T^\al)\|\sum_{k=1}^N\tau A_h(T) E_{h,\tau}^{N-k}(N) (A_h(T)-A_h(t_k))e_k\|_{L^2\II}.
\end{equation*}
Then by applying the \textsl{a priori} estimate in Lemma \ref{lem:ful:a-priori}
and the same argument in Theorems \ref{thm:stab-small} and \ref{thm:stab-big}, we derive
$$\|e_0\|_{L^2\II}  \le c_T \tau |\log\tau| \|\tuh(0)\|_{L^2\II}.$$
Finally, the Lemmas \ref{lem:reg-back} and \ref{lem:semi-err}
leads to the desired result.
\end{proof}

%

Now we are ready to state the main theorem showing the error of the numerical reconstruction from noisy data.
The proof is a direct result of  Lemma \ref{lem:err-reg0}, \ref{lem:semi-err}, \ref{lem:err-ful-noise} and \ref{lem:err-reg:ful}.
\begin{theorem}\label{thm:fully-err}
Let $U_{h,\gamma}^{0,\delta}$ be the numerical reconstructed initial data 
using the fully discrete scheme \eqref{eqn:back-fully},
and $u_0$ be the exact initial data.
Suppose Assumption \ref{ass} is valid. Then if $\| u_0 \|_{\dH q}\le c$ with $q\in (0,2]$ there holds 
\begin{equation*}
\|U_{h,\gamma}^{0,\delta}-u_0\|_{L^2\II}\le c\Big(\gamma^\frac q2+ \delta\gamma^{-1} +
h^2\gamma^{-1}+\tau {|\log\tau|} (h^2\gamma^{-1}+1)  \Big) \end{equation*}
Moreover, for $u_0\in L^2\II$, there holds
\begin{equation*}
\| U_{h,\gamma}^{0,\delta} -  u(0)\|_{L^2\II} \rightarrow 0\quad \text{as}~~\delta,\gamma,h,\tau\rightarrow0,
~~\frac\delta\gamma\rightarrow0~~\text{and}~~ \frac{h^2}{\gamma} \rightarrow0.
\end{equation*}
\end{theorem}

The \textsl{a priori} error estimate in Theorem \ref{thm:fully-err} give a useful guideline to choose the regularization
parameter $\gamma$ and the discretization parameters $h$ and $tau$ according to the noise level $\delta$.
In particular, if $u_0 \in \dH q$, by choosing
\begin{equation*}
\gamma \sim \delta^{\frac{2}{q+2}},~~h\sim \delta^\frac 12~~\text{and}~~\tau |\log \tau| \sim \delta^{\frac{q}{q+2}},
\end{equation*}
we obtain the optimal approximation error
\begin{equation*}
\|U_{h,\gamma}^{0,\delta} - u(0)\|_{L^2\II}\le c\delta^{\frac{q}{q+2}}.
\end{equation*}
 
\section{Numerical Experiments}\label{sec:5}

Now we test several two dimensional examples with $\Omega=(0,1)^2$ in order to illustrate our theoretical results. 
Throughout the section, we apply the standard Galerkin piecewise linear  FEM with uniform mesh size $h=1/(M+1)$ for the space discretization, and the backward Euler convolution quadrature method with uniform mesh size $\tau = T/N$ for time discretization. We solve the direct problem to obtain the exact observation data by using fine meshes, i.e. $h=1/100$, $\tau = T/500$. Then we compute the noisy observational data by
\begin{equation*}
g_\delta = u(T)+ \varepsilon\delta\sup_{x\in\Omega} u(x,T)
\end{equation*}
where $\varepsilon$ is generated from standard Gaussian distribution and $\delta$ denotes the related noisy level.

We begin with the following time-dependent diffusion coefficient:
\begin{equation*}
a_1(x,y,t) = \begin{pmatrix}
y\sin((1+t)^{0.5})+2 & -0.1 \\
-0.1 & \sin(\pi x)(t+1.2)^{-0.8}+2
\end{pmatrix},
\end{equation*} 
satisfying conditions \eqref{Cond-1-t}-\eqref{Cond-2-t} and Assumption \ref{ass:large-t}. 
We solve the linear system \eqref{eqn:back-fully} by using  the conjugate gradient method.
\vskip5pt
\paragraph{\bf Smooth initial data}
We begin with a smooth initial data:
$$u_0 = \sin(2\pi x)\sin(2\pi y) \in \dH2.$$
According to Theorem \ref{thm:fully-err}, 
we compute $U_{h,\gamma}^{0,\delta}$ with $\gamma \sim \sqrt{\delta}$ and $h,\tau \sim \sqrt{\delta}$, and expect a convergence of order $O(\sqrt{\delta})$. 
Numerical results presented in Figure \ref{fig:fully:smooth:err1} 
fully support the theoretical result. On the other hand, our numerical results indicate that 
the recovery is stable for all $T$, might be neither very large nor very small.
This interesting phenomenon warrants further investigation in the future.
In Figure \ref{fig:profile:case1}, we present profiles of solutions and errors  with different noise level.
\begin{figure}[htbp]
\centering
\begin{subfigure}{.33\textwidth}
\centering
\includegraphics[scale=0.33]{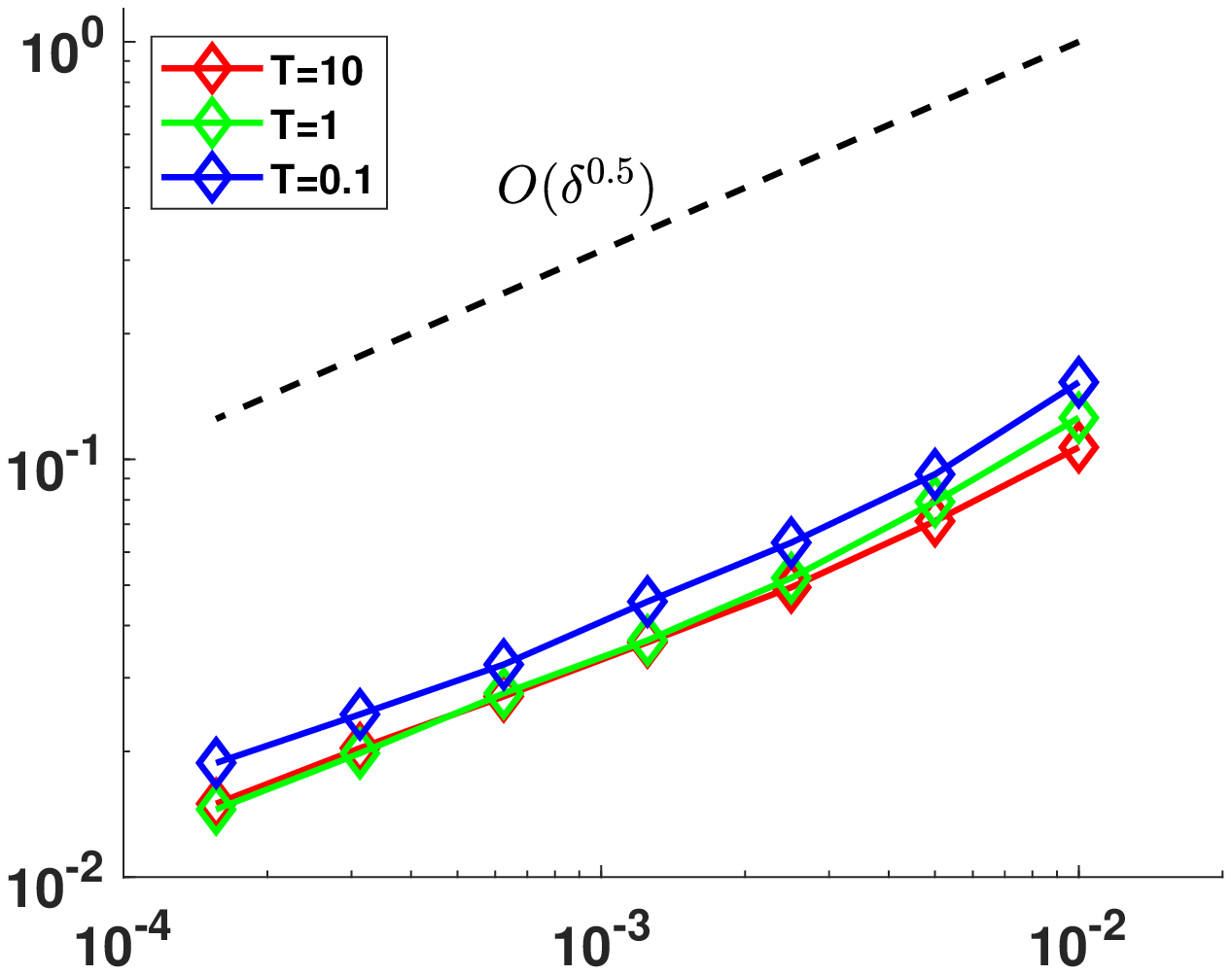}
\caption{$\alpha = 0.25$.}
\end{subfigure}%
\begin{subfigure}{.33\textwidth}
\centering
\includegraphics[scale=0.33]{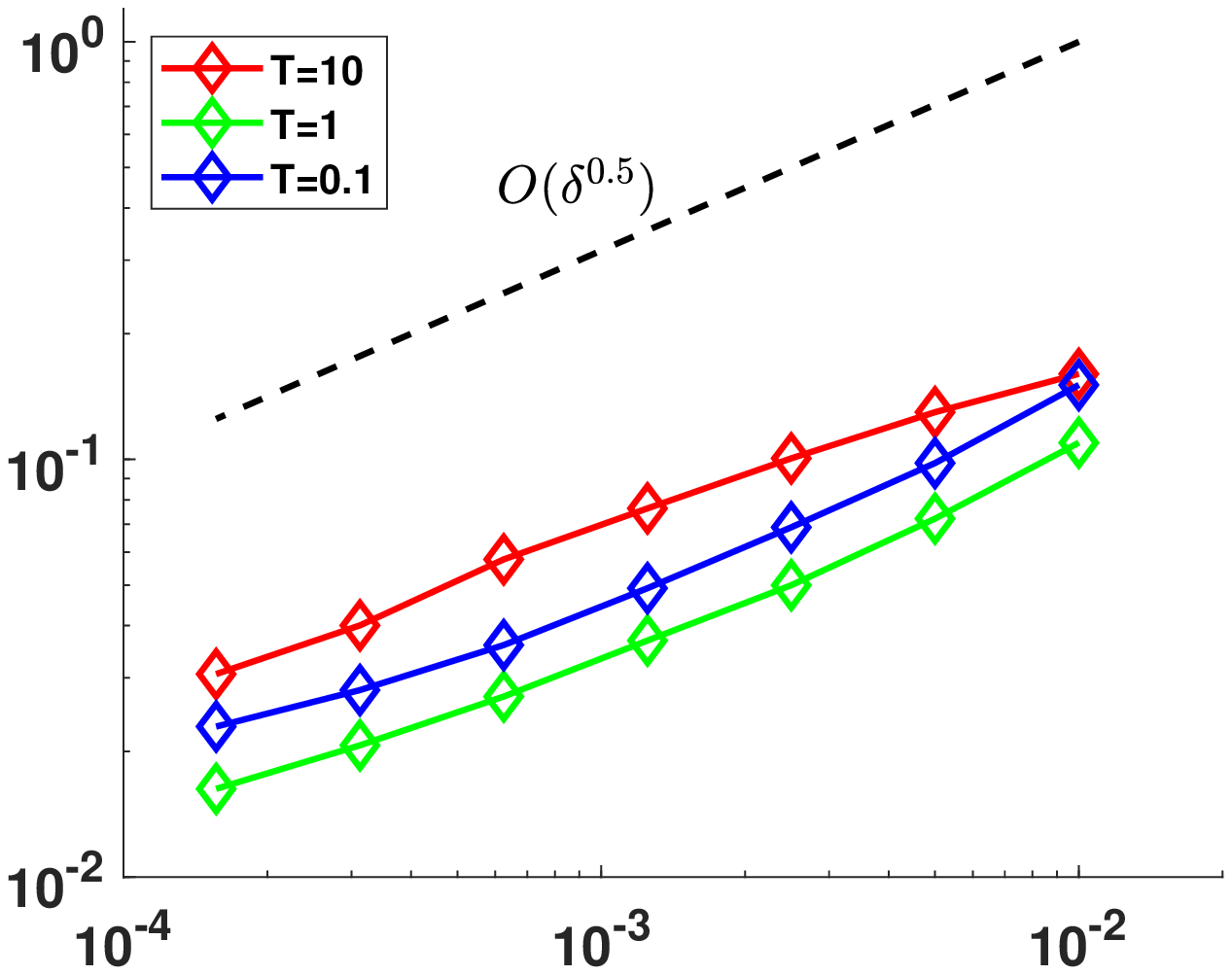}
\caption{$\alpha = 0.5$.}
\end{subfigure}%
\begin{subfigure}{.33\textwidth}
\centering
\includegraphics[scale=0.33]{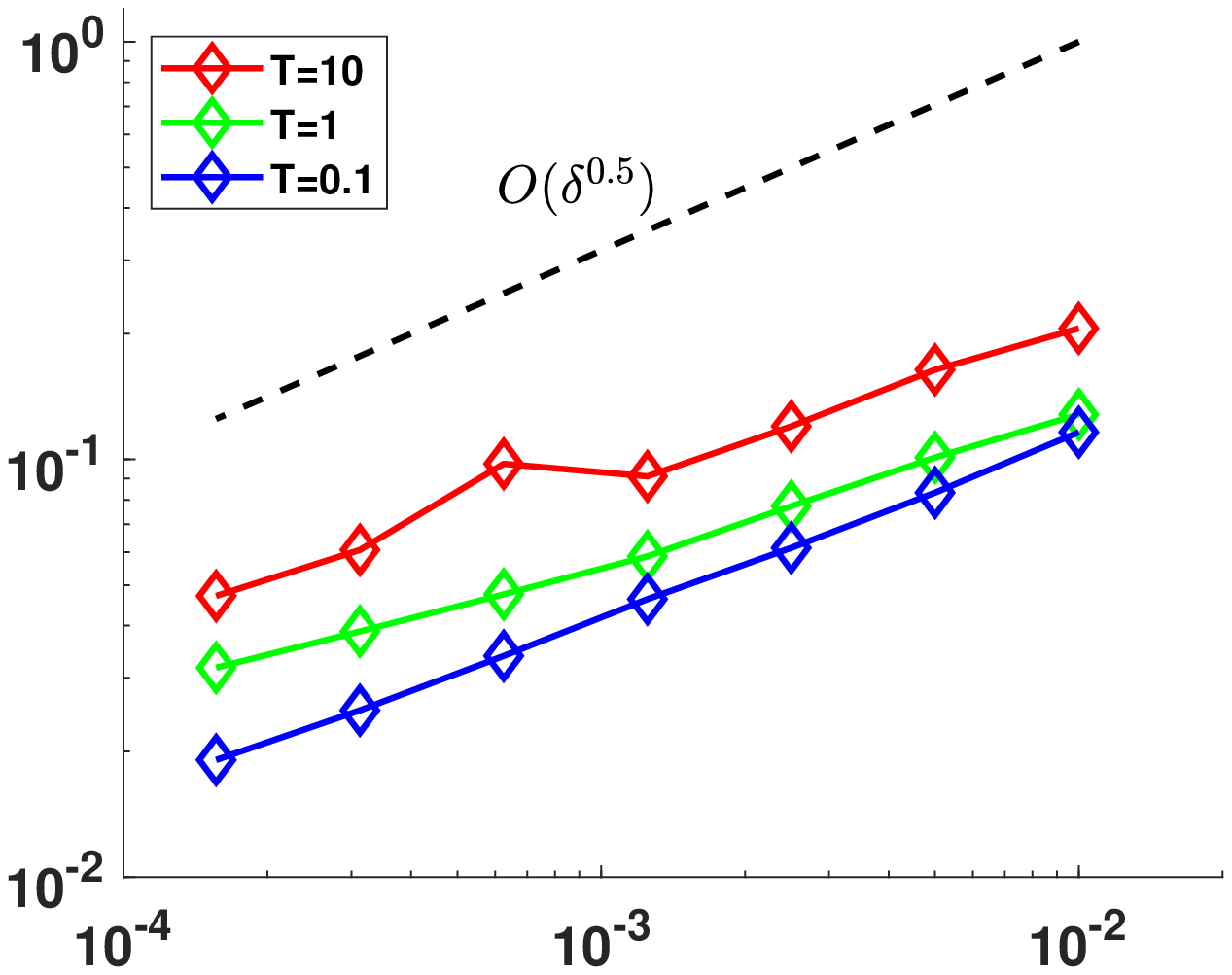}
\caption{$\alpha =0.75$.}
\end{subfigure}%
\caption{Plot of error: $a_1(x,t)$ and smooth initial data; $h=\sqrt{\delta}$, $\tau = \sqrt{\delta}/5$, $\gamma = \sqrt{\delta}/350$.}
\label{fig:fully:smooth:err1}
\end{figure}

\begin{figure}[htbp]
\begin{subfigure}{.24\textwidth}
\centering
\includegraphics[scale=0.25]{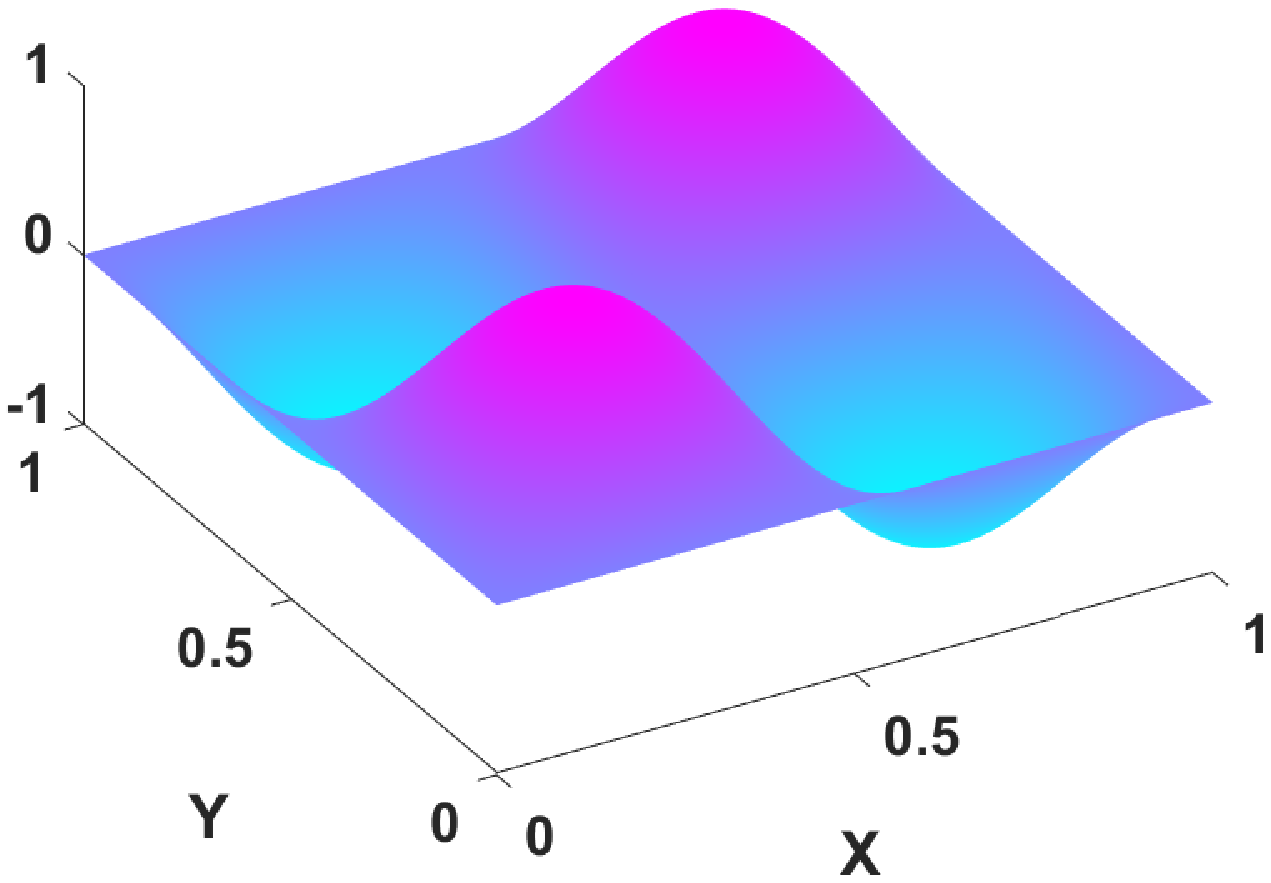}
\end{subfigure}%
\begin{subfigure}{.24\textwidth}
\centering
\includegraphics[scale=0.25]{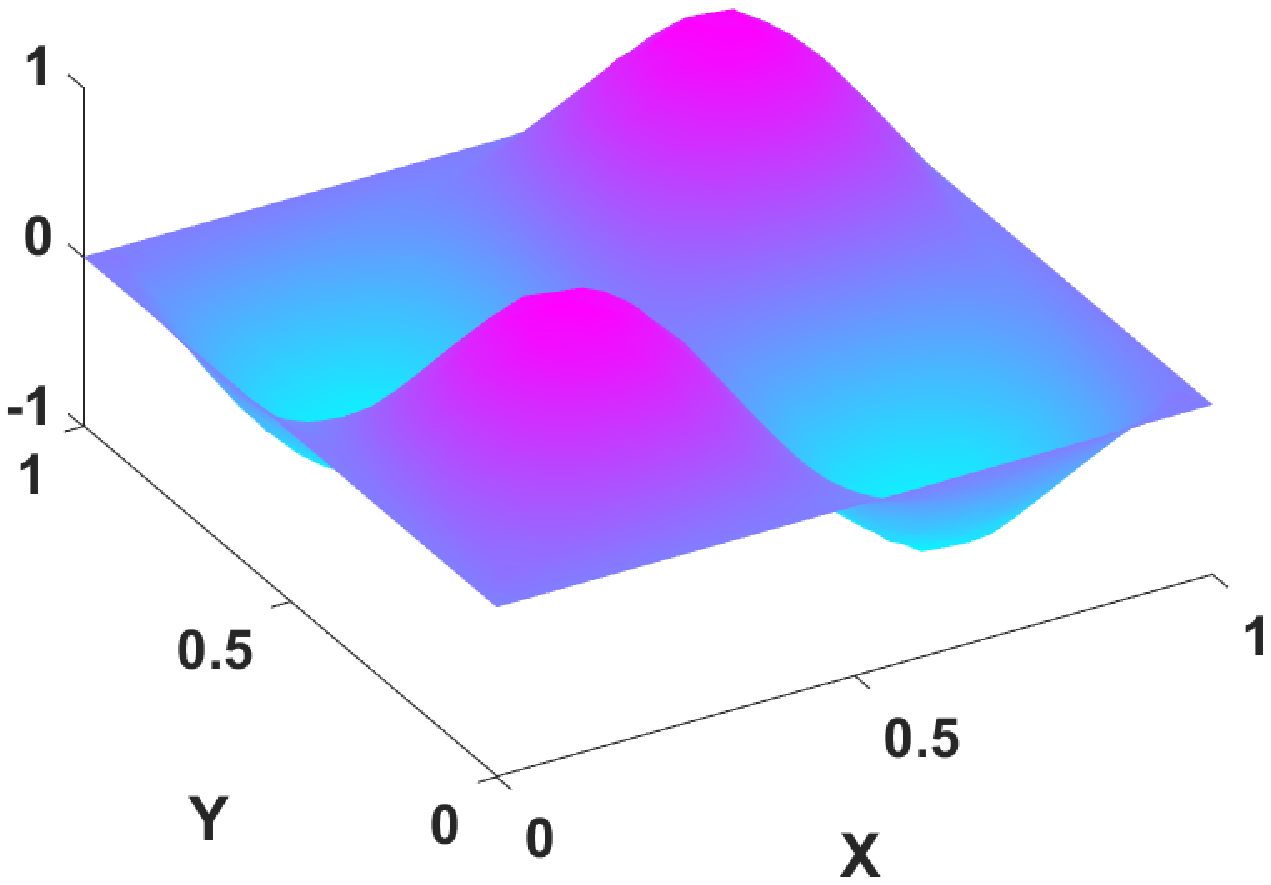}
\end{subfigure}%
\begin{subfigure}{.24\textwidth}
\centering
\includegraphics[scale=0.25]{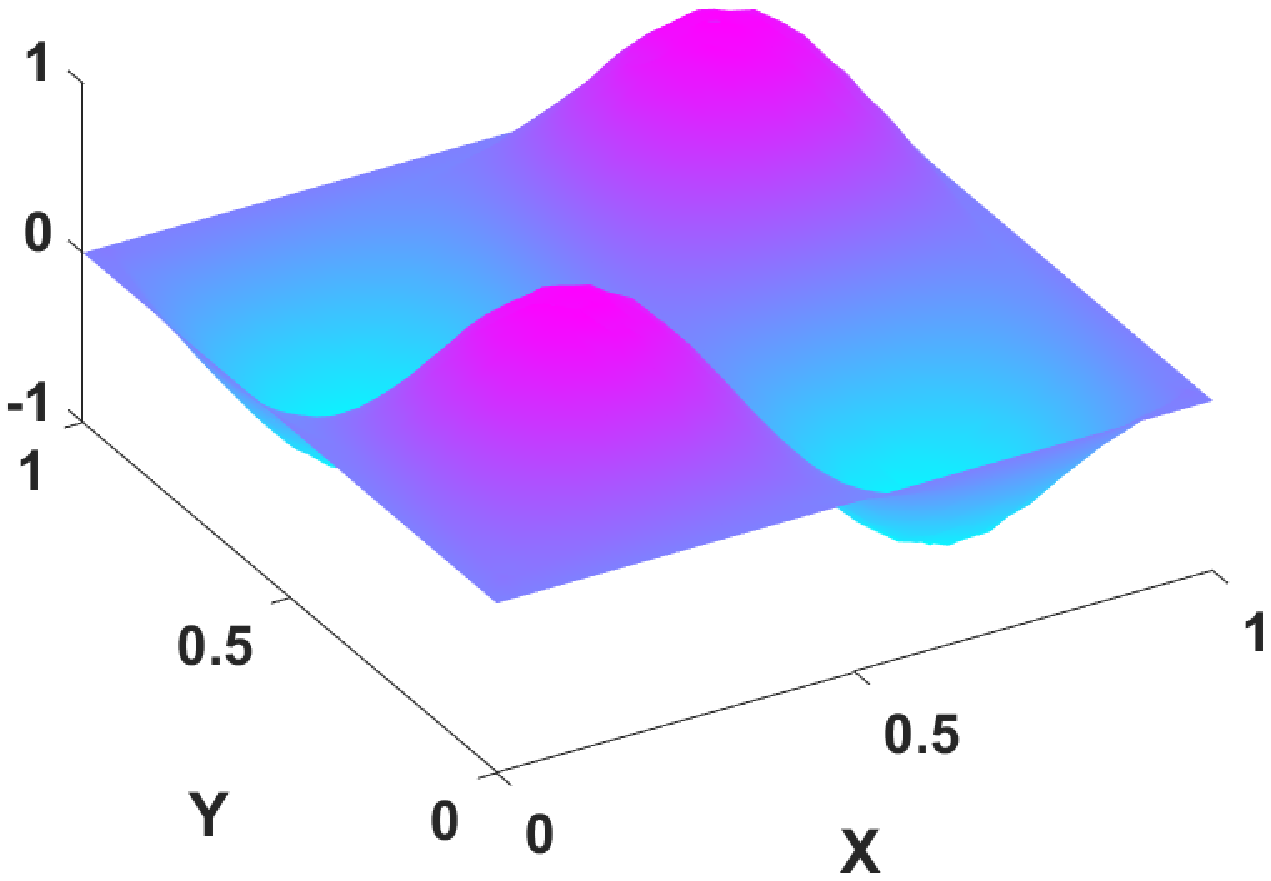}
\end{subfigure}
\begin{subfigure}{.24\textwidth}
\centering
\includegraphics[scale=0.25]{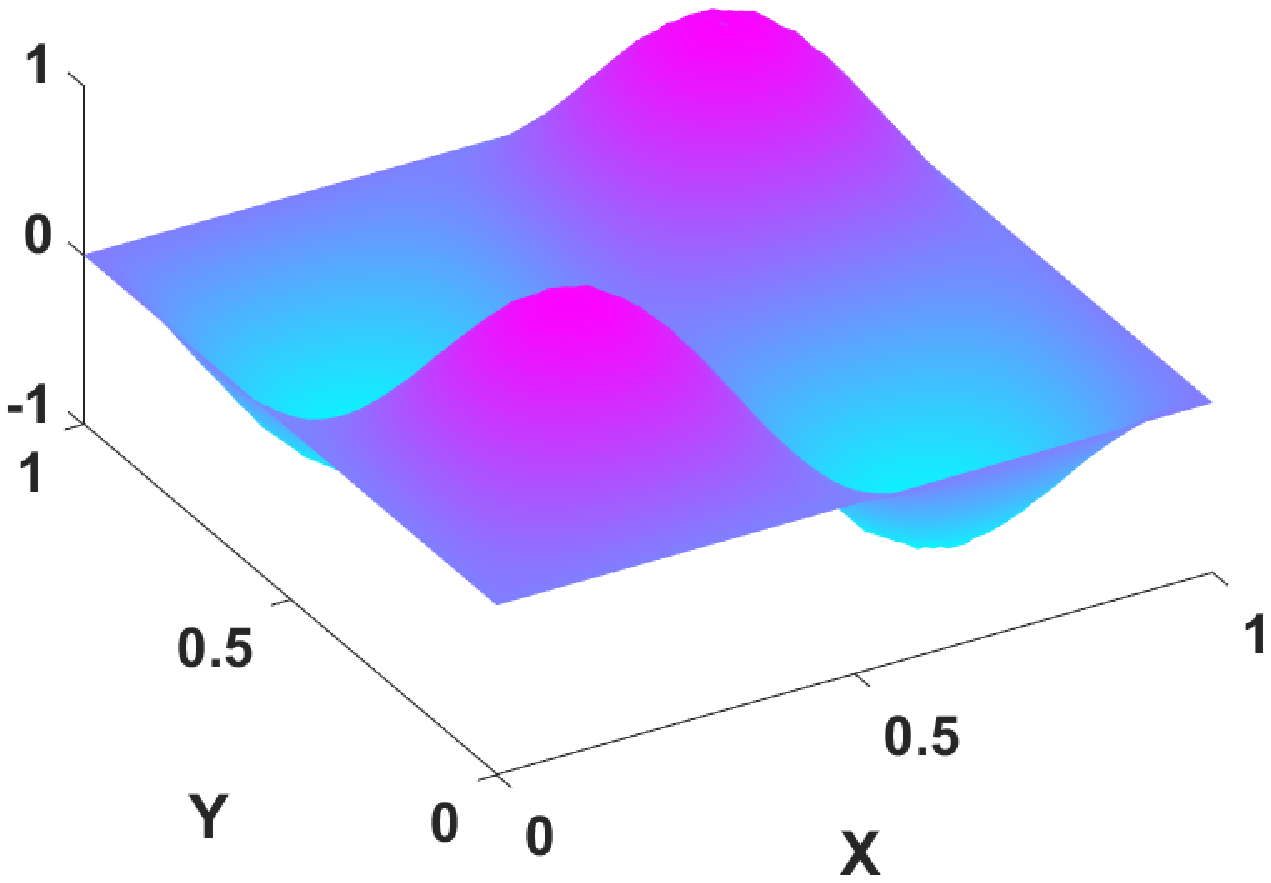}
\end{subfigure}
\newline
\raggedleft
\begin{subfigure}{.24\textwidth}
\centering
\includegraphics[scale=0.25]{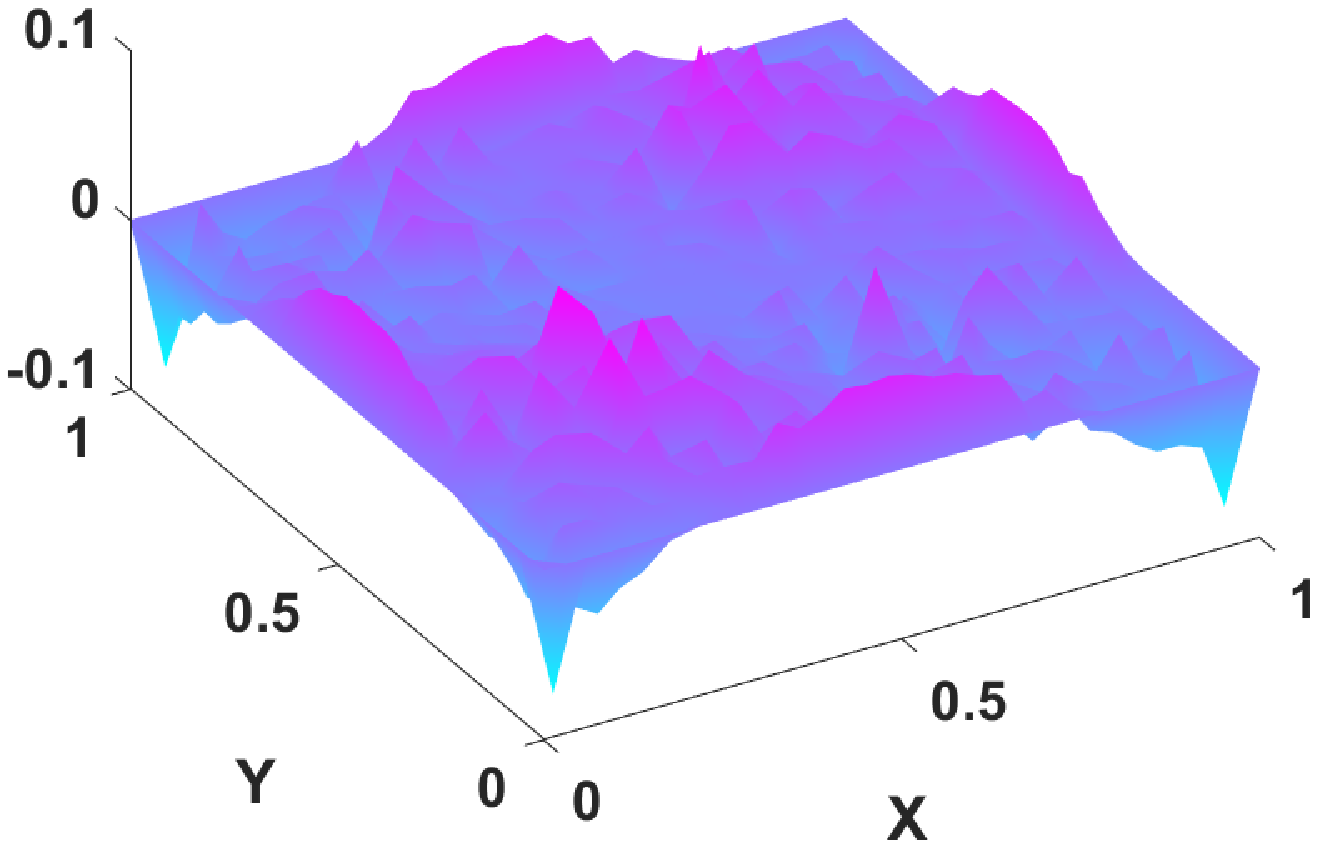}
\caption{$\delta =1e-2$.}
\end{subfigure}%
\begin{subfigure}{.24\textwidth}
\centering
\includegraphics[scale=0.25]{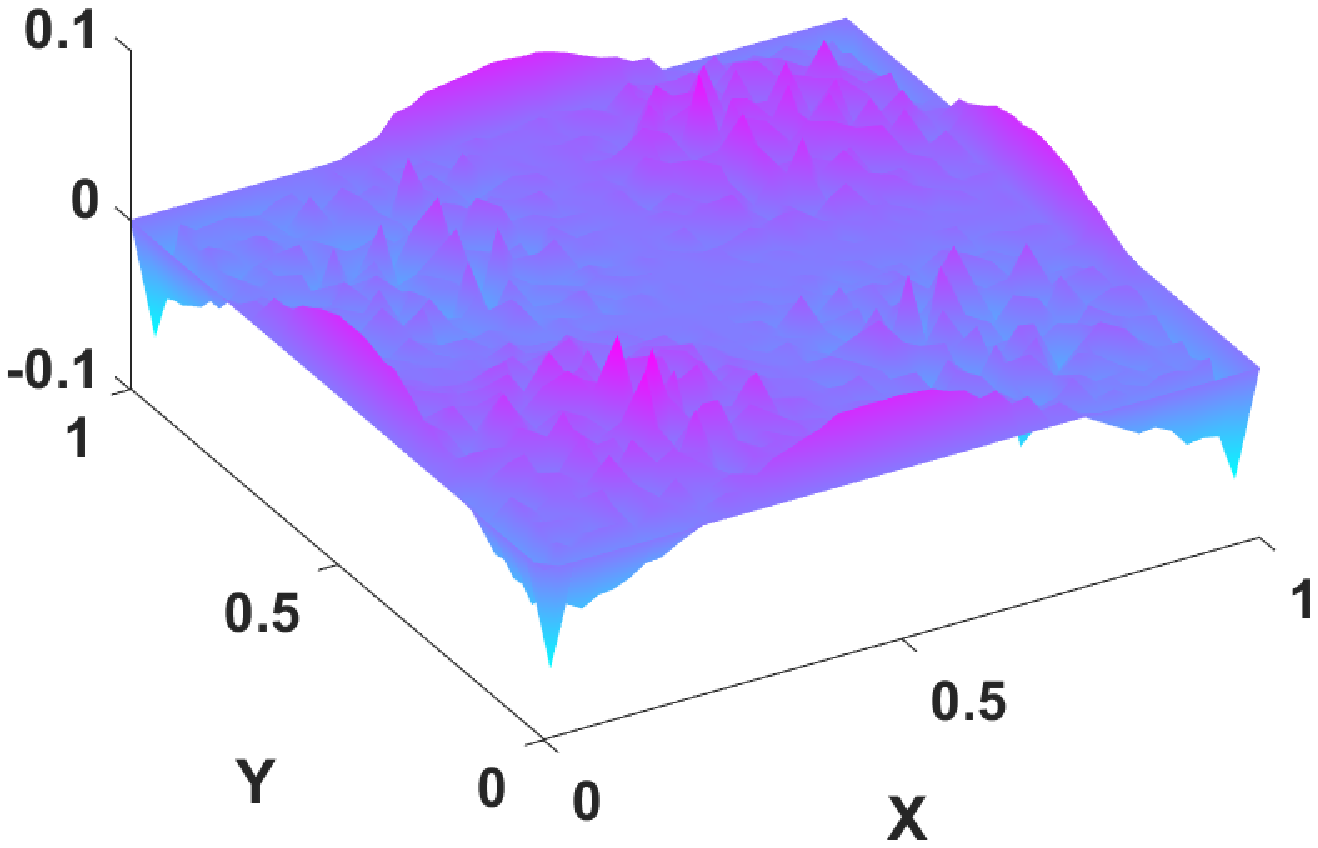}
\caption{$\delta = 5e-3$.}
\end{subfigure}
\begin{subfigure}{.24\textwidth}
\centering
\includegraphics[scale=0.25]{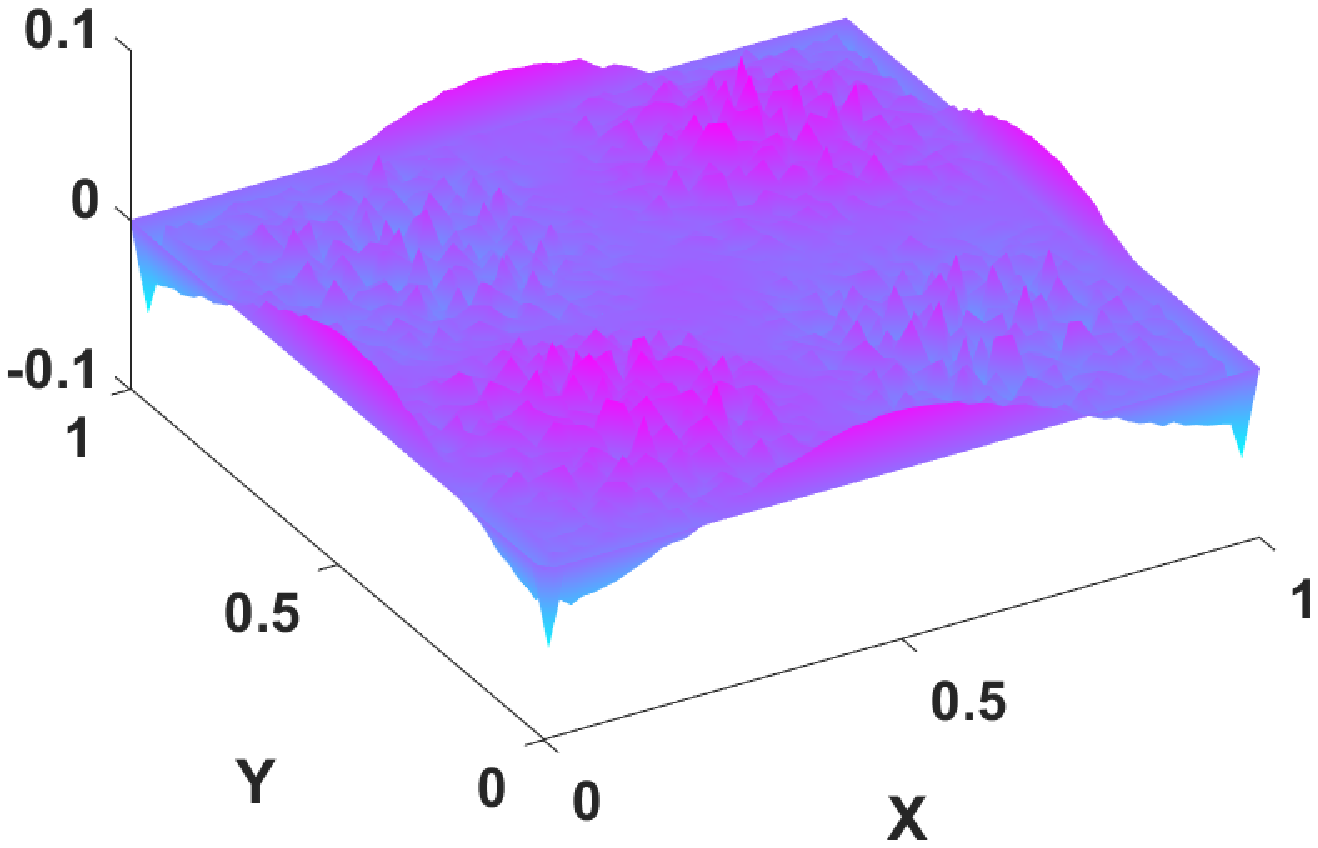}
\caption{$\delta= 2.5e-3$.}
\end{subfigure}
\caption{Profiles of Top left: Exact initial data $u_0$. Recover with $a_1(x,t)$, $\alpha =0.5$, $T=1$. The remain three columns are profiles of numerical reconstructions $\tU_0$ and theirs errors, with $h=\sqrt{\delta}$, $\tau = \sqrt{\delta}/5$, $\gamma=\sqrt{\delta}/350$.}
\label{fig:profile:case1}
\end{figure} 

\paragraph{\textbf{Nonsmooth initial data.}} In this example we consider the following nonsmooth initial condition
\begin{equation*}
u_0 = \begin{cases}
1,~~\text{if}~0.5 \le x \le 1,\\
0,~~\text{otherwise}
\end{cases}
\end{equation*}
Note that $u_0 \in \dH {\frac12-\varepsilon}$  for any $\varepsilon\in (0,\frac 12)$.
Then Theorem \ref{thm:fully-err} indicate that the optimal convergence rate  is almost $O(\delta^{0.2})$ provided that 
$\gamma =O(\delta^{0.8})$, $h = O(\sqrt{\delta})$ and $\tau = O(\delta^{0.2})$.
This is fully supported by the numerical results presented in Figure \ref{fig:fully:err:case3}. 
In Figure \ref{fig:profile:case3} we plot the profiles of solutions and errors, which also confirm that the numerical recovery is reliable.
\begin{figure}[htbp]
\centering
\begin{subfigure}{.33\textwidth}
\centering
\includegraphics[scale=0.33]{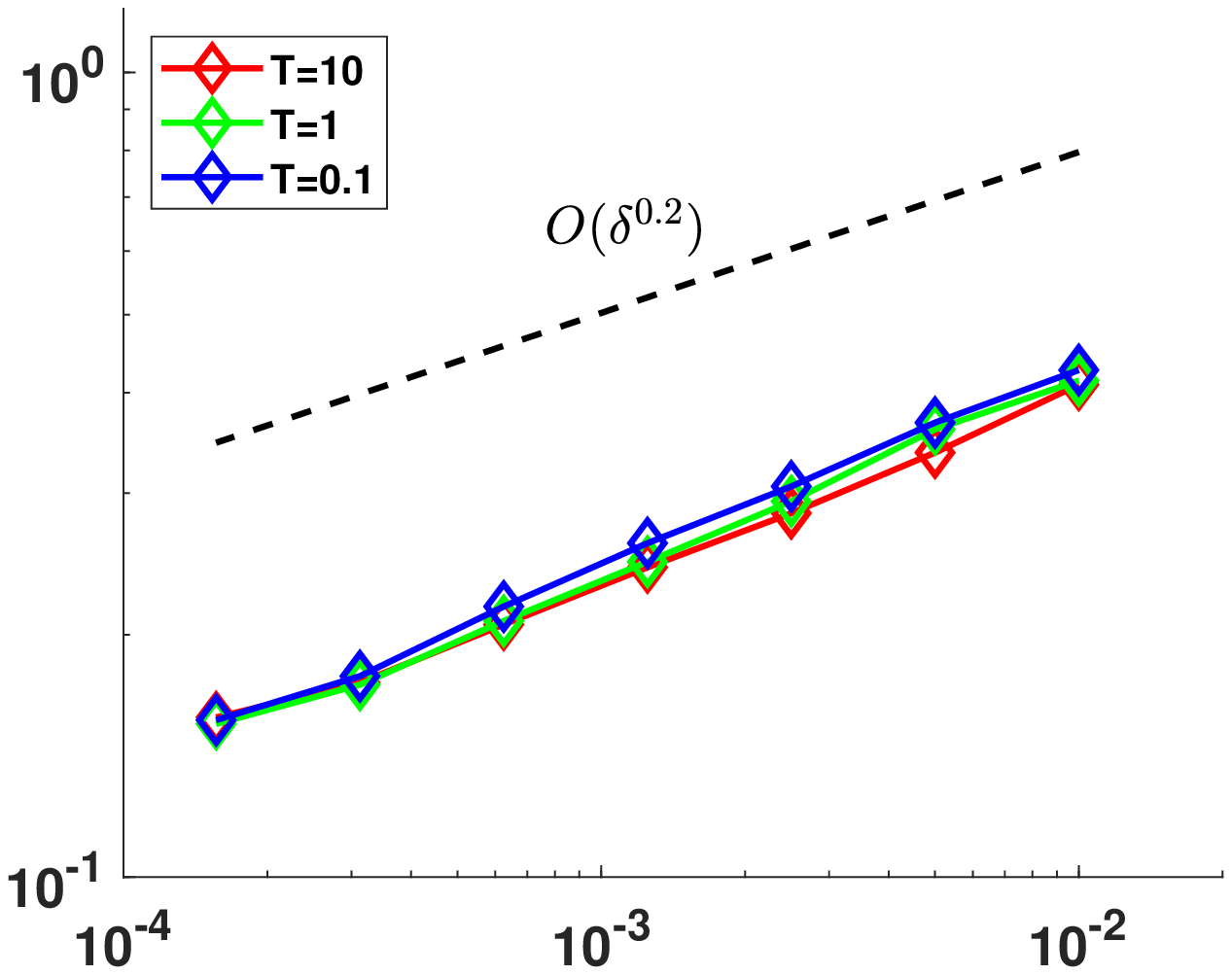}
\caption{$\alpha = 0.25$.}
\end{subfigure}%
\begin{subfigure}{.33\textwidth}
\centering
\includegraphics[scale=0.33]{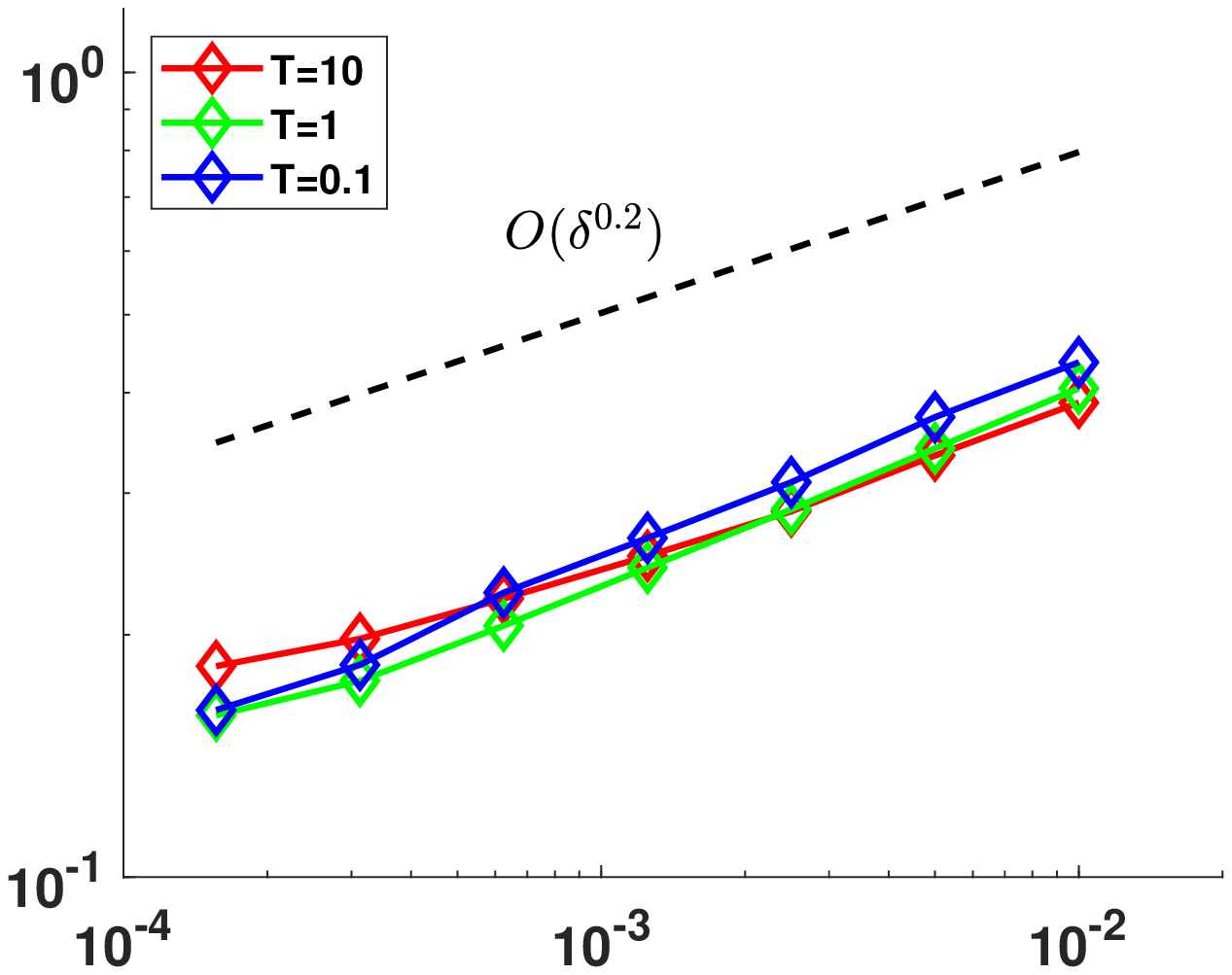}
\caption{$\alpha = 0.5$.}
\end{subfigure}%
\begin{subfigure}{.33\textwidth}
\centering
\includegraphics[scale=0.33]{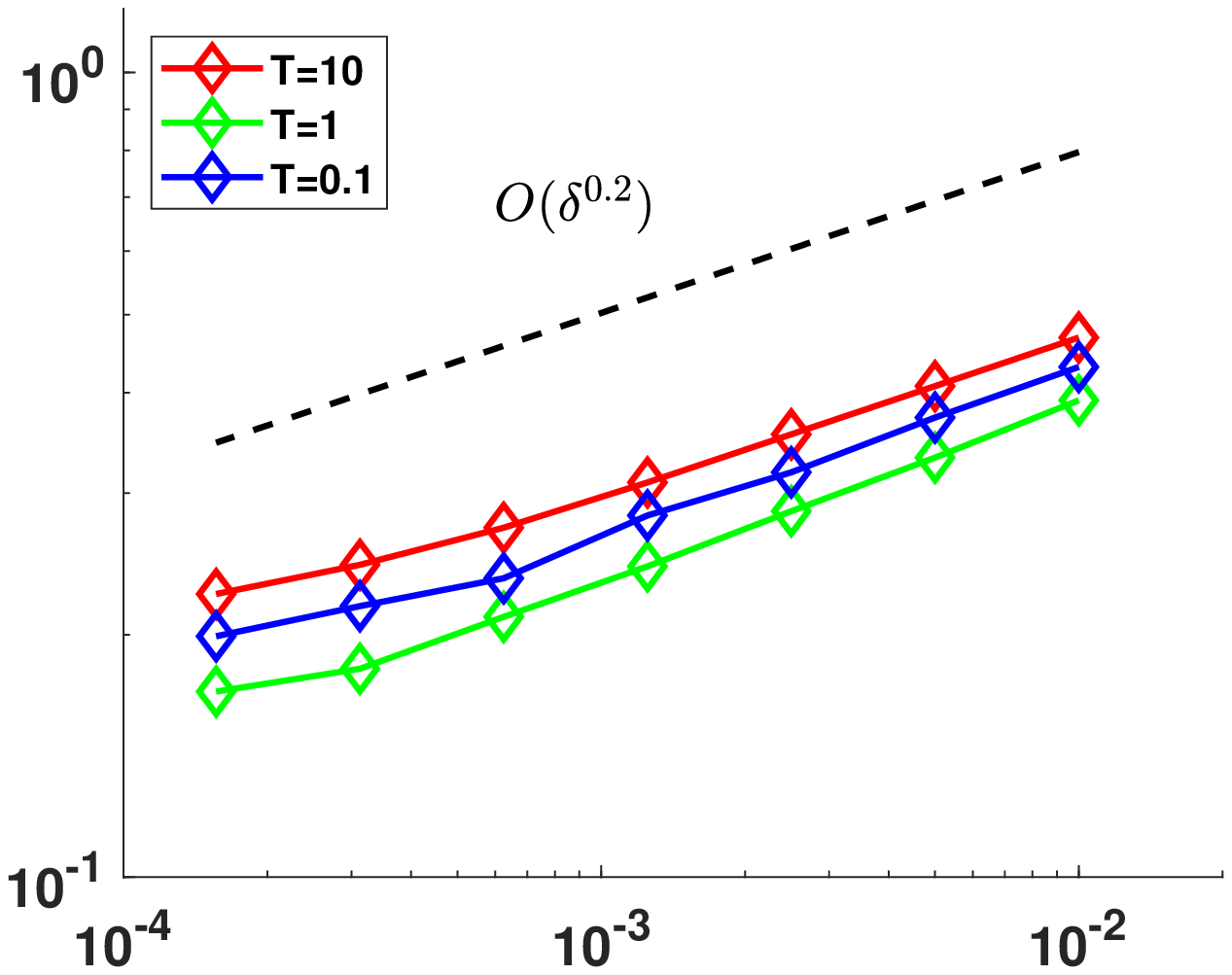}
\caption{$\alpha =0.75$.}
\end{subfigure}%
\caption{Plot of error: $a_1(x,t)$ and smooth initial data; $h=\sqrt{\delta}$, $\tau = \delta^{0.2}/20$, $\gamma = \delta^{0.8}/200$.}
\label{fig:fully:err:case3}
\end{figure}
\begin{figure}[htbp]
\begin{subfigure}{.24\textwidth}
\centering
\includegraphics[scale=0.25]{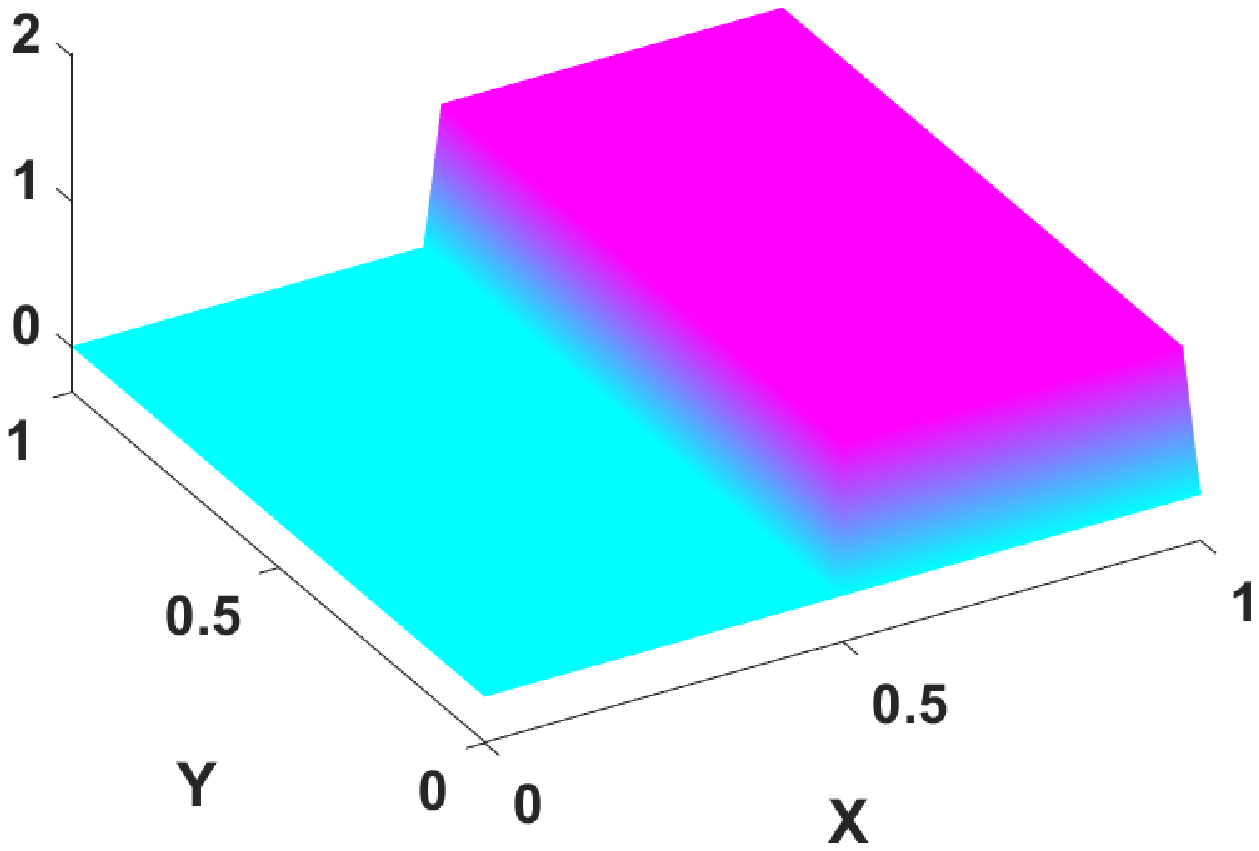}
\end{subfigure}%
\begin{subfigure}{.24\textwidth}
\centering
\includegraphics[scale=0.25]{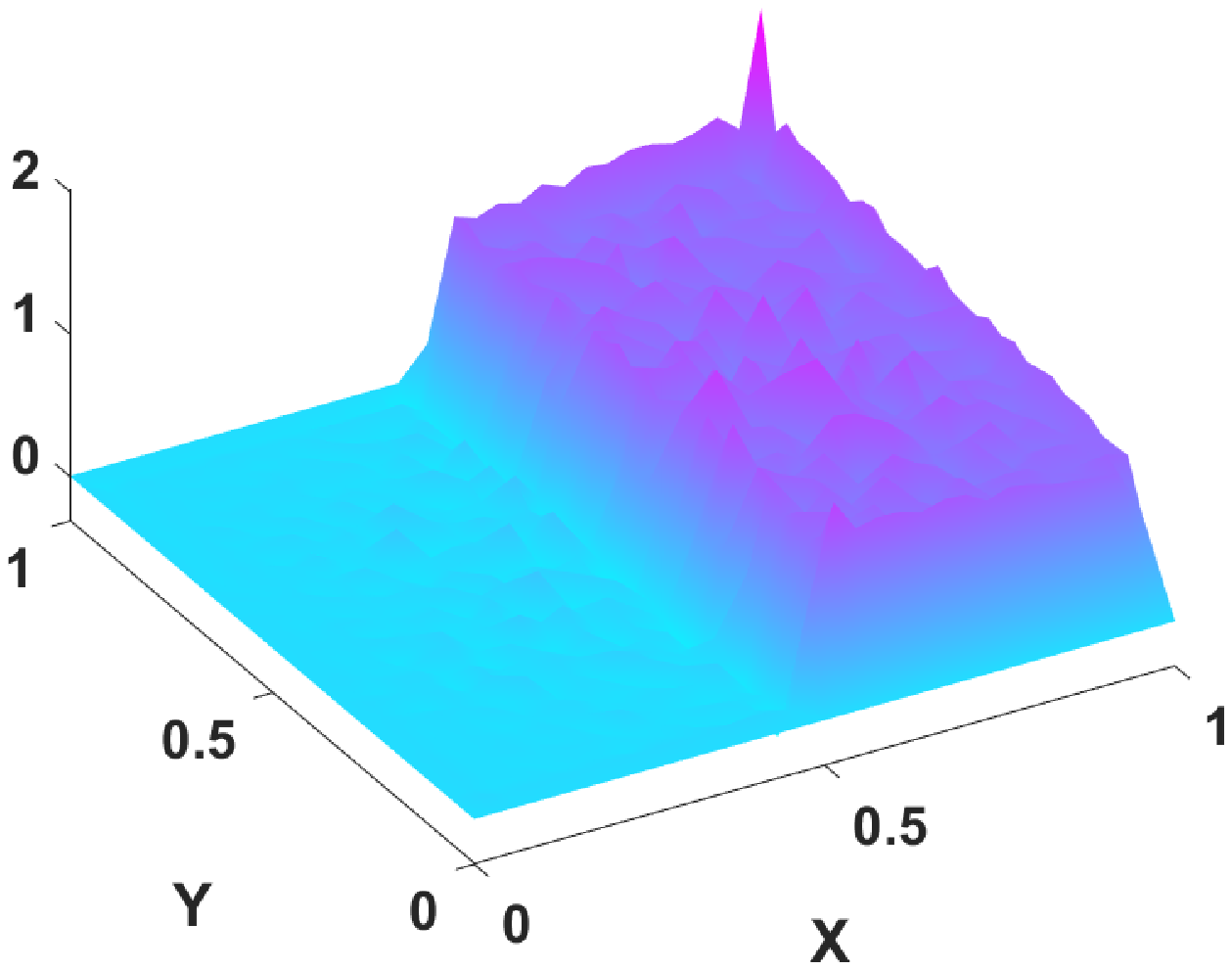}
\end{subfigure}%
\begin{subfigure}{.24\textwidth}
\centering
\includegraphics[scale=0.25]{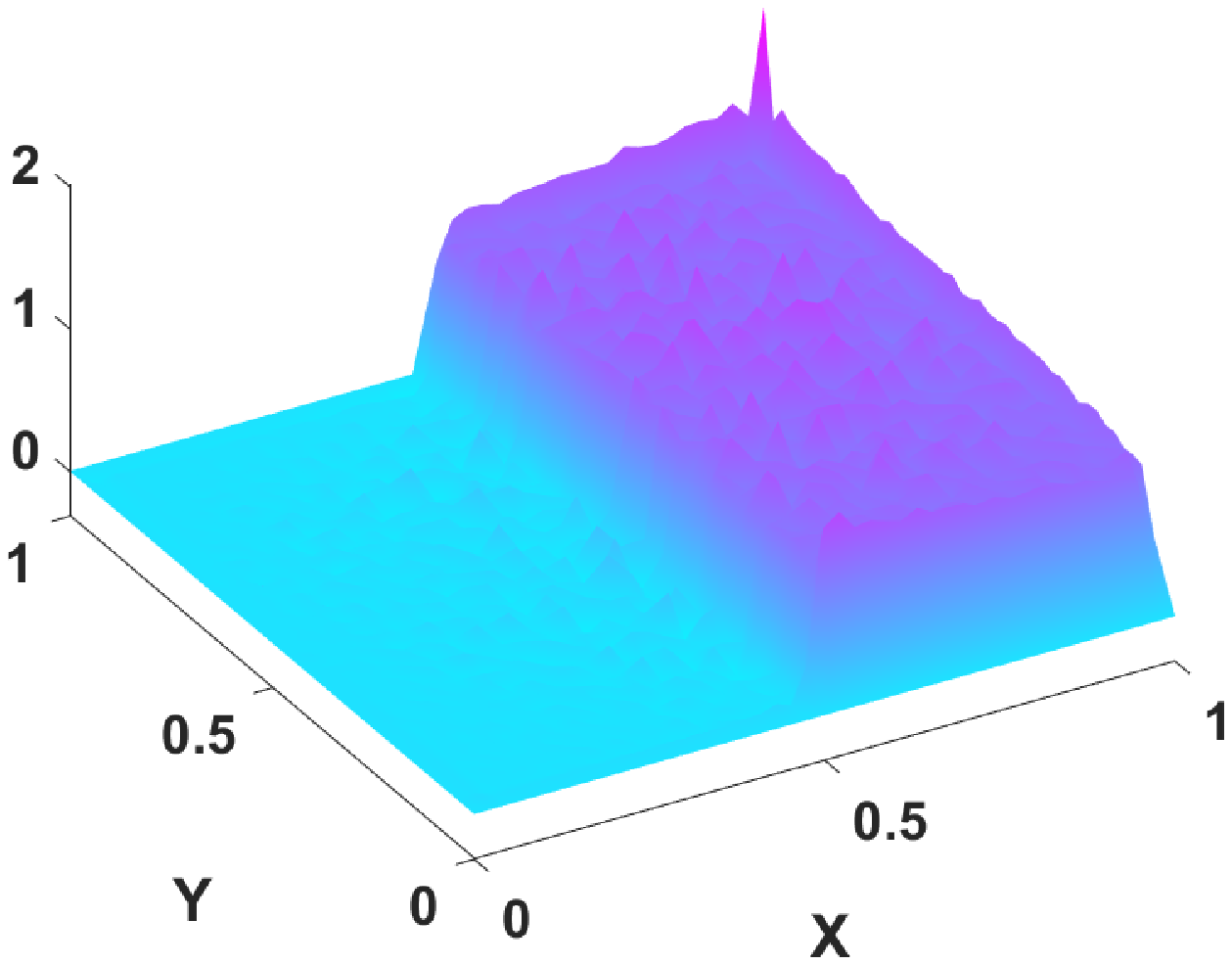}
\end{subfigure}
\begin{subfigure}{.24\textwidth}
\centering
\includegraphics[scale=0.25]{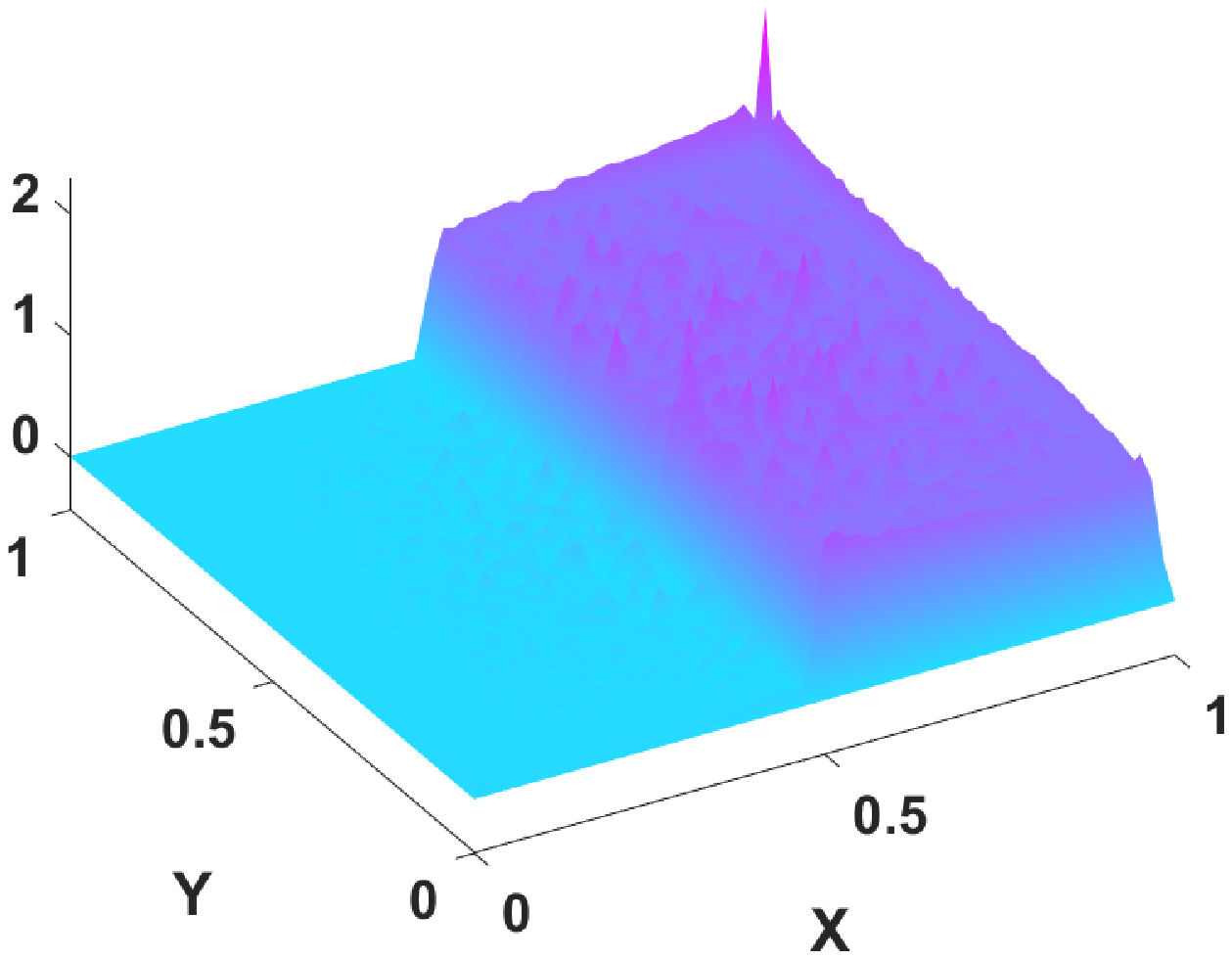}
\end{subfigure}
\newline
\raggedleft
\begin{subfigure}{.24\textwidth}
\centering
\includegraphics[scale=0.25]{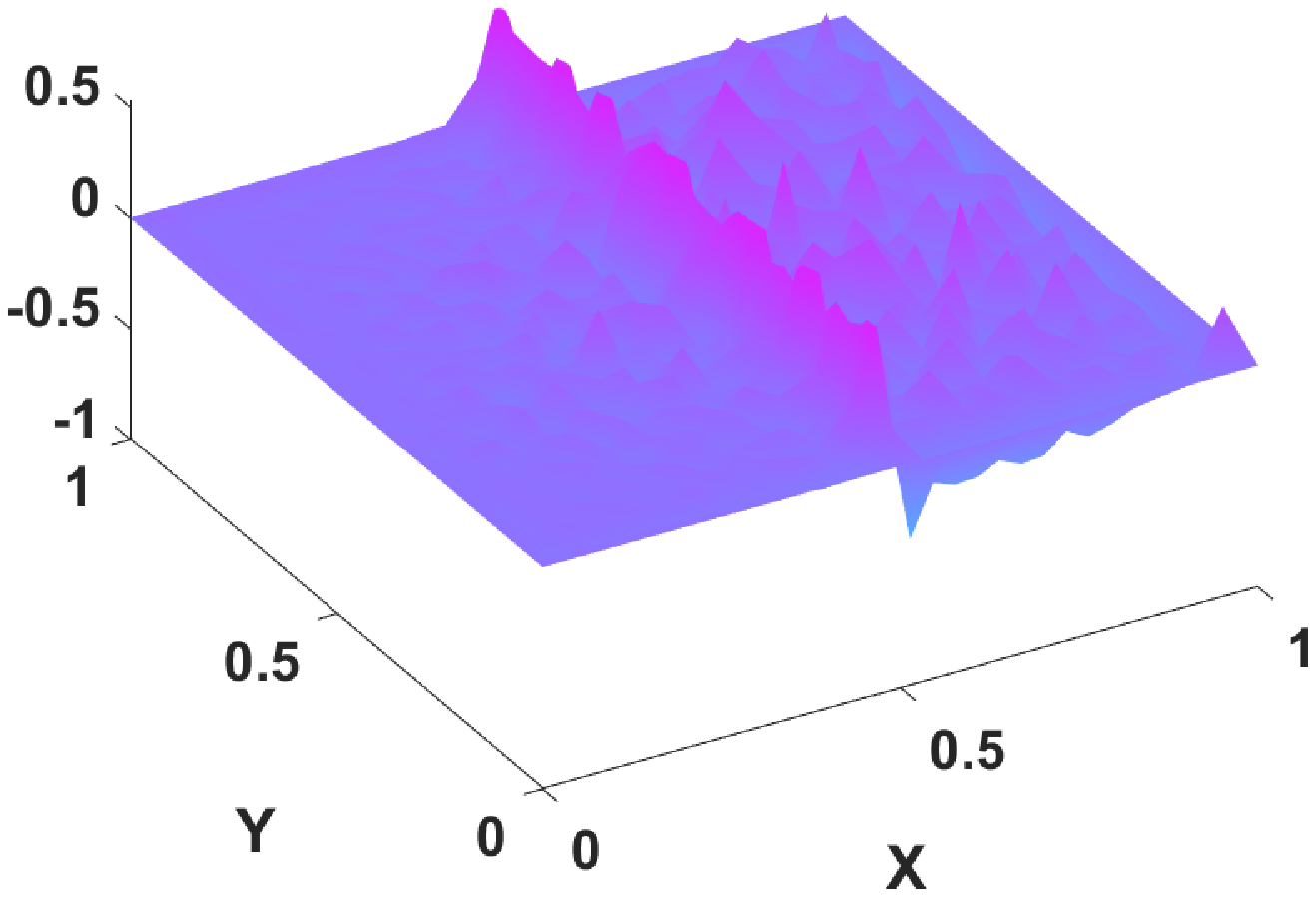}
\caption{$\delta =1e-2$.}
\end{subfigure}%
\begin{subfigure}{.24\textwidth}
\centering
\includegraphics[scale=0.25]{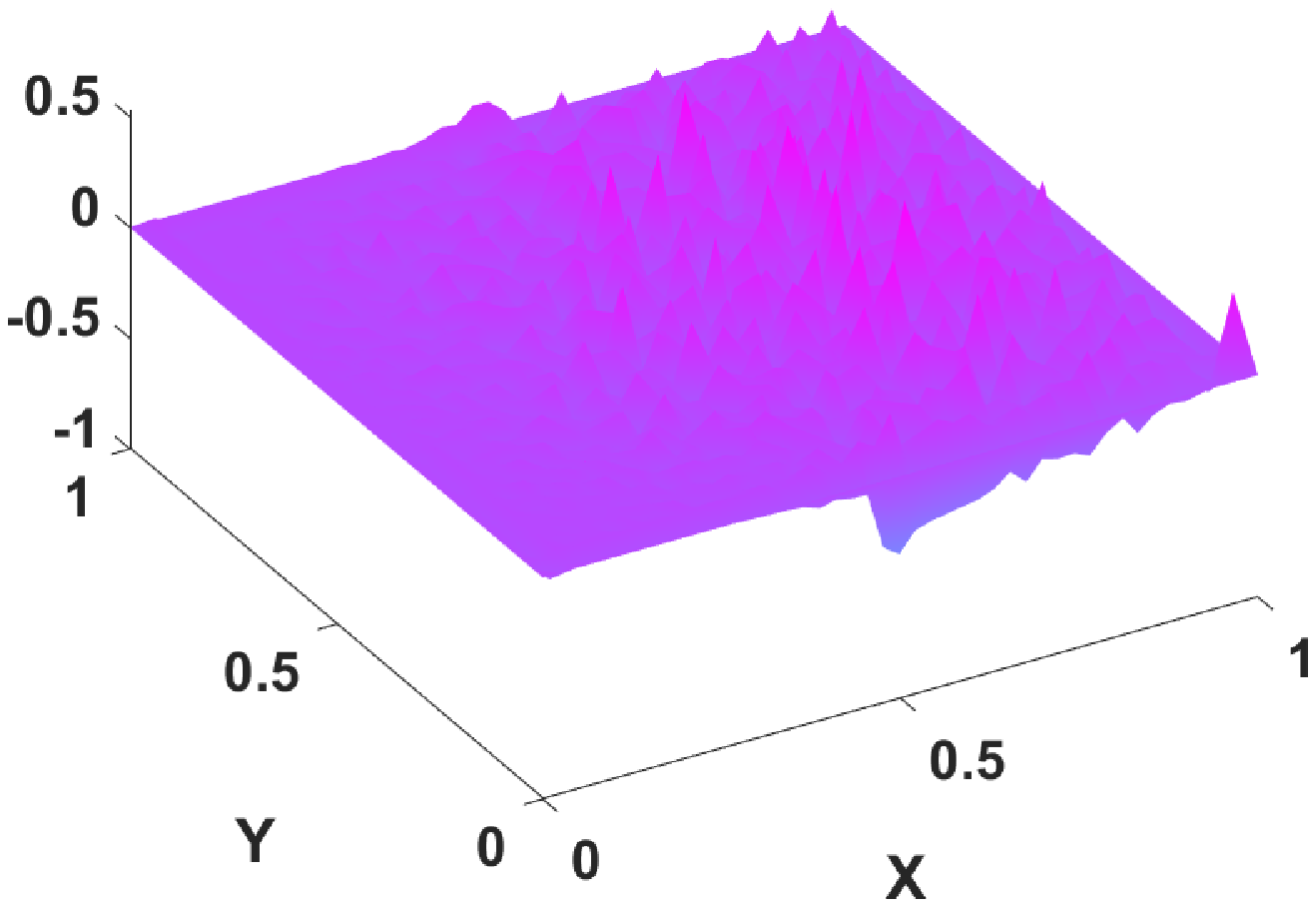}
\caption{$\delta = 5e-3$.}
\end{subfigure}
\begin{subfigure}{.24\textwidth}
\centering
\includegraphics[scale=0.25]{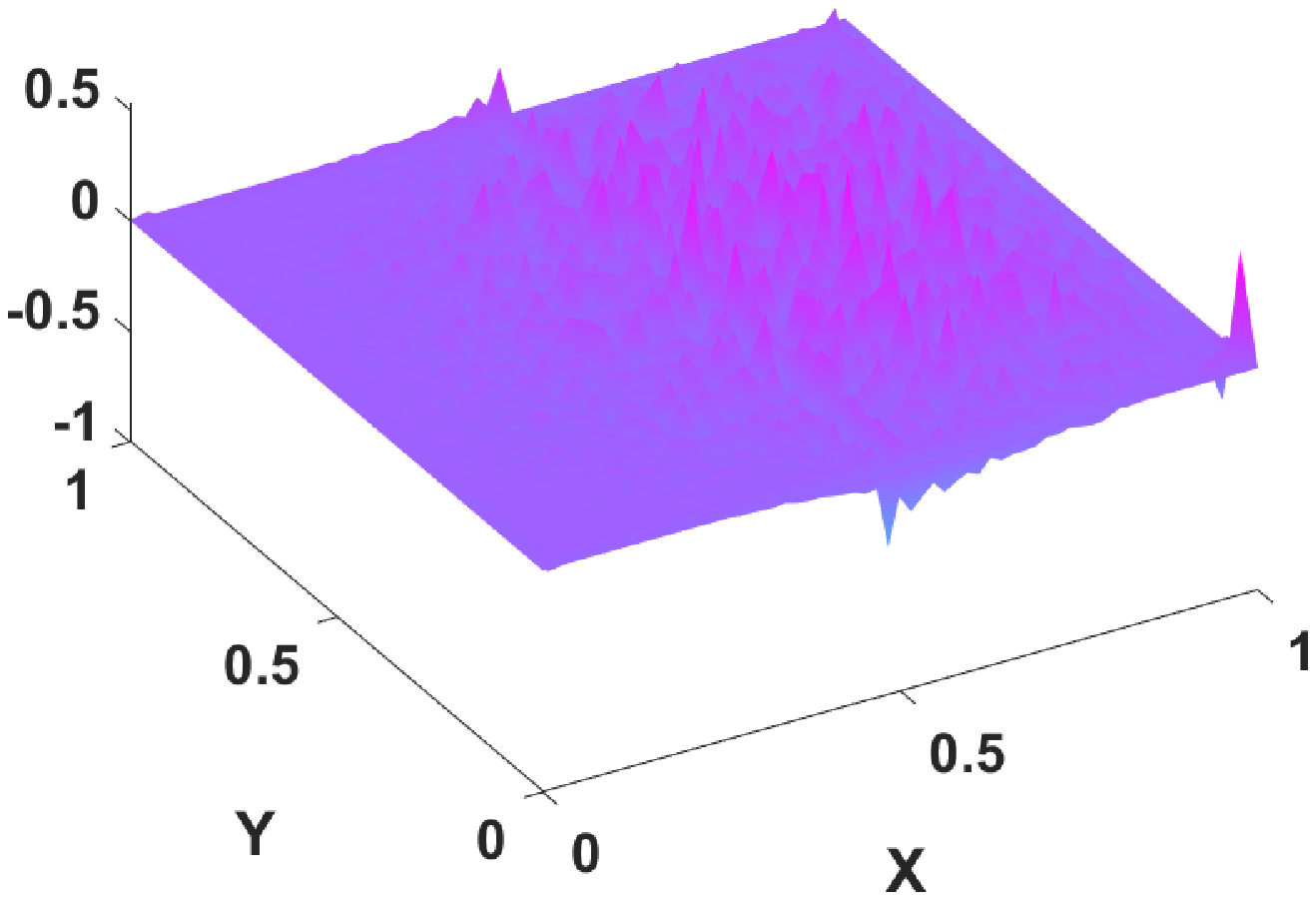}
\caption{$\delta= 2.5e-3$.}
\end{subfigure}
\caption{Top left: Exact initial data $u_0$. Recover with $a_1(x,t)$, $\alpha =0.5$, $T=1$. The remain three columns are profiles of numerical reconstructions $\tU_0$ and theirs errors, with $h=\sqrt{\delta}$, $\tau = \delta^{0.2}/20$, $\gamma=\delta^{0.8}/200$.}
\label{fig:profile:case3}
\end{figure} 
\paragraph{\textbf{Example violating Assumption \ref{ass:large-t}.}}
We also test the following diffusion coefficient
\begin{equation*}
a_2(x,y,t) = \begin{pmatrix}
e^{-x}\cos(t)+2 & (1.5-(t+1)^{-0.2})/10\\
(1.5-(t+1)^{-0.2})/10 & \cos(\pi y)\sin(t)+2
\end{pmatrix}.
\end{equation*}
Note that $a_2$ satisfies conditions \eqref{Cond-1-t} and \eqref{Cond-2-t}, but 
Assumption \ref{ass:large-t} is not fullfilled. 

Numerical experiments show that the numerical reconstruction via the fully discrete scheme \eqref{eqn:back-fully} 
still converges under proper parameter choices.
For example, we test the smooth initial data $u_0 = \sin(2\pi x)\sin(2\pi y)$ and large terminal time $T=10$. 
We choose $\gamma,h,\tau\sim\sqrt{\delta}$, and observe a convergence rate around $O(\sqrt{\delta})$, see cf. Figure \ref{fig:fully:err:case2}.
We will continue to consider the general case in our future studies.
\begin{figure}
\centering
\includegraphics[scale=0.5]{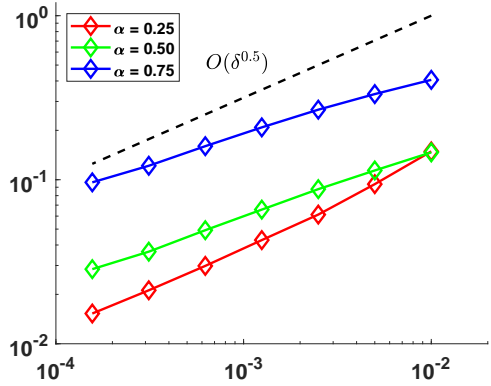}
\caption{Plot of error: $a_2(x,t)$ and smooth initial data; $T=10$, with $h=\sqrt{\delta}$, $\tau = \delta^{0.5}/5$, $\gamma = \delta^{0.5}/350$ for $\alpha=0.25, 0.5$ and $\gamma = \delta^{0.5}/150$ for $\alpha=0.75$.}
\label{fig:fully:err:case2}
\end{figure}
%

\section*{Acknowledgement}
This work is partly supported by Hong Kong Research Grants Council (Project No. 15304420) and 
an internal grant of Hong Kong Polytechnic University (Project ID: P0031041, Work Programme: ZZKS).

\section*{Appendix}\label{sec:appendix} \textbf{Appendix A. Proof of Lemma \ref{lem:cond-AT}} \label{append-A}\vskip5pt
For $p=0$,  Conditions \eqref{Cond-1-t} and \eqref{Cond-2-t} and Assumption \ref{ass:large-t} imply
\begin{align*}
\|(A(t)-A(s))v\|_{L^2\II} &\le c\big(\|\nabla(a(t)-a(s))\|_{L^\infty\II}+\|a(t)-a(s)\|_{L^\infty\II}\big)\|v\|_{\dH2}\\
&\le c \min(1,\min(t,s)^{-\kappa}|t-s|)\|v\|_{\dH 2}.
\end{align*}
For $p=-2$, from using the duality argument, we have 
\begin{align*}
\|(A(t)-A(s))v\|_{\dH{-2}} &= \sup_{\fy\in \dH 2} \frac{\langle(A(t)-A(s))v,\fy\rangle}{\|\fy\|_{\dH 2}} = \sup_{\fy\in \dH 2} \frac{(v,(A(t)-A(s))\fy)}{\|\fy\|_{\dH 2}}  \\
&  \le   \sup_{\fy\in \dH 2}  \frac{\| v \|_{L^2\II} \| (A(t)-A(s))\fy  \|_{L^2\II}}{\|\fy\|_{\dH 2}}\\
&\le  c \min(1,\min(t,s)^{-\kappa}|t-s|)  \sup_{\fy\in \dH 2}\| v \|_{L^2\II}
\end{align*}
This completes the proof of the lemma.\vskip15pt

\textbf{Appendix B. Proof of Lemma \ref{lem:ful:a-priori}} \label{append-B}

Recalling the fact that \cite[Lemma 3.3]{JinZhou:2021sicon}
$$ U_h^n\bDal U_h^n  \geq \frac 12 \bDal  |U_h^n|^2 $$ 
Therefore like Lemma \ref{lemma:prior-estimate} we define an operator $\underline{A_h}= -c_0 \Delta_h$.
Condition \ref{Cond-1-t} gives that  the operator $A_h(t) - \underline{A_h}$
is selfadjoint and positive semidefinite for all $n\ge 1$. Rewrite the equation  \eqref{eqn:fully} as
$$  \bDal (U_h^n-U_h^0) +  \underline{A_h} U_n =  (\underline{A_h}-A_h(t)) U_h^n  \quad \text{for all}~~  1\le n\le N.  $$
Taking inner product with $U_h^n$ on the above equation and by definition of $-\Delta_h$ and $A_h(t)$ , we obtain
$$  (\bDal (U_h^n-U_h^0), U_h^n) +  c_0 \| \nabla U_h^n \|_{L^2\II}^2
=  \big((c_0-a(\cdot,t)) \nabla U_h^n, \nabla U_h^n\big)\le 0 \quad \text{for all}~~  1\le n\le N.  $$
Using the above inequality and Poincar\'e inequality we arrive at\begin{align*}
&\bDal(\|U_h^n\|_{L^2\II}-\|U_h^0\|_{L^2\II})+c\|U_h^n\|_{L^2\II}\\
&\le   \bDal\Big[( \|U_h^n\|_{L^\II}-\|U_h^0\|_{L^2\II})(1+\|U_h^0\|_{L^2\II}/\|U_h^n\|_{L^2\II}) \Big] +  c  \|U_h^n\|_{L^2\II}   \\
&\le 0 \quad \text{for all}~~ n\ge 1, 
\end{align*} 
for some constant $c$ uniform in $t_n$. Then  the comparison principle for discrete fractional ODEs \cite{LiWang:2019} leads to
$$  \|U_h^n\|_{L^2\II} \le F_\tau^n(c)\| U_h^0 \|_{L^2\II} \le \frac{c}{1+ct_n^\alpha} \| U_h^0\|_{L^2\II}. $$
where the definition of  $F_\tau^n(c)$  can be found in \cite{ZhangZhou:2020, ZhangZhou:2021}.  This immediately leads to the desired result.

Next by solution representation  \eqref{eqn:repre-U} we have 
\begin{align*}
&\|A_h(t_{n_*})U_h^n\|_{L^2\II}\\
& \le  \|A_h(t_{n_*})F_{h,\tau}^n (n_*)U_h^0\|_{L^2\II} + \tau\| \sum_{k=1}^n A_h(t_{n_*})E_{h,\tau}^{n-k} (n_*)(I-A_h(t_k)A_h(t_{n_*})^{-1})A_h(t_{n_*})U_h^k\|_{L^2\II}\\
&\le c t_n^{-\al} \|U_h^0\|_{L^2\II} +   \sum_{k=1}^n \|\tau A_h(t_{n_*})E_{h,\tau}^{n-k} (n_*)(I-A_h(t_k)A_h(t_{n_*})^{-1})\|\,\|A_h(t_{n_*})U_h^k\|_{L^2\II},
\end{align*}
lemma \ref{lem:op:fully} and \ref{lem:cond-Ah} show that  
\begin{align*}
\|A_h(t_{n_*})U_h^n\|_{L^2\II}\le c t_n^{-\al} \|U_h^0\|_{L^2\II} +   \sum_{k=1}^n c\tau\|A_h(t_{n_*})U_h^k\|_{L^2\II},
\end{align*}
the discrete version of Gronwall's inequality \cite[Lemma 10.5]{Thomee:2006} gives that 
\begin{equation*}
\|A_h(t_{n_*})U_h^n\|_{L^2\II}\le c \exp(ct_n) t_n^{-\al} \|U_h^0\|_{L^2\II} 
\end{equation*}
here $c$ is uniform in $n$ , $\tau$ and $t_n$.

Meanwhile, in the other hand  $\| I - A(t_*)^{-1} A(s) \| \le c |t_*-s|^{\beta}$ for any $\beta \in [0,1]$.
Then if $\beta=(1+\epsilon)\alpha$ with $\epsilon\in(0,1/\alpha-1)$ we can derive that 
\begin{align*}
&\|A_h(t_{n_*})U_h^n\|_{L^2\II}\\
& \le  \|A_h(t_{n_*})F_{h,\tau}^n (n_*)U_h^0\|_{L^2\II} + \tau\| \sum_{k=1}^n A_h(t_{n_*})E_{h,\tau}^{n-k} (n_*)(I-A_h(t_k)A_h(t_{n_*})^{-1})A_h(t_{n_*})U_h^k\|_{L^2\II}\\
&\le c t_n^{-\al} \|U_h^0\|_{L^2\II} +   \sum_{k=1}^n \|\tau A_h^2(t_{n_*})E_{h,\tau}^{n-k} (n_*)(I-A_h(t_{n_*})^{-1}A_h(t_k))\|\,\|U_h^k\|_{L^2\II}\\
&\le c t_n^{-\al} \|U_h^0\|_{L^2\II} +c\tau \sum_{k=1}^n (t_{n_*}-t_k)^{-1+\varepsilon \al} t_k^{-\al} \|U_h^0\|_{L^2\II}\\
&\le  c t_n^{-\al} \|U_h^0\|_{L^2\II} +c\int_0^{t_n}(t_{n_*}-s)^{-1+\epsilon \al}s^{-\al}\d s \|U_h^0\|_{L^2\II}\le ct_n^{-(1-\epsilon)\al} \|U_h^0\|_{L^2\II}
\end{align*}

\bibliographystyle{abbrv}

\end{document}